\documentclass[11pt]{amsart}
\usepackage{epsfig}

\usepackage{graphicx}
\usepackage{amscd}
\usepackage{amsmath}
\usepackage{amsfonts}
\usepackage{yfonts}
\usepackage{psfrag}
\usepackage{amssymb}
\usepackage{verbatim}
\usepackage{comment}

\newtheorem{theorem}{Theorem}[section]
\theoremstyle{plain}

\newtheorem{corollary}[theorem]{Corollary}

\newtheorem{defi}[theorem]{Definition}

\newtheorem{lemma}[theorem]{Lemma}

\def\Int{{\rm int}}
\def\freq{{\rm freq}}
\def\diam{{\rm diam}}
\def\dist{{\rm dist}}
\def\supp{{\rm supp}}

\def\Hk{{\mathcal H}}

\def\th{\theta}

\def\wtil{\widetilde}

\newcommand{\lam}{\lambda}

\newcommand{\gam}{\gamma}
\newcommand{\om}{\omega}

\def\Om{\Omega}

\newcommand{\Gam}{\Gamma}

\newcommand{\R}{{\mathbb R}}
\newcommand{\Q}{{\mathbb Q}}
\newcommand{\Z}{{\mathbb Z}}
\newcommand{\C}{{\mathbb C}}

\def\N{{\mathbb N}}

\newcommand{\Nat}{{\mathbb N}}

\def\Ak{{\mathcal A}}
\def\Bk{{\mathcal B}}
\def\Rk{{\mathcal R}}
\def\Qk{{\mathcal Q}}
\def\Sk{{\mathcal S}}
\def\Lk{{\mathcal L}}
\def\Mk{{\mathcal M}}
\def\Ik{{\mathcal I}}

\def\Pk{{\mathcal P}}
\def\Ck{{\mathcal C}}
\def\Sf{{\sf S}}
\def\Tk{{\mathcal T}}

\def\be{\begin{equation}}
\def\ee{\end{equation}}

\newcommand{\Fk}{{\mathcal F}}

\newcommand{\eps}{{\varepsilon}}
\newcommand{\es}{\emptyset}

\def\ov{\overline}
\newcommand{\const}{{\rm const}}

\def\Frs{{\mathfrak S}}

\begin{document}

\title[Limit theorems for self-similar tilings]{Limit theorems for self-similar tilings}

\author{Alexander I. Bufetov}
\address{Alexander I. Bufetov\\ Steklov Institute, Moscow; 
The Institute for Information Transmission Problems, Moscow; 
National Research University Higher School of Economics, Moscow, Russia; 
Rice University, Houston, TX, USA
}
\email{bufetov@mi.ras.ru}
\author{Boris Solomyak }
\address{Boris Solomyak, Box 354350, Department of Mathematics,
University of Washington, Seattle, WA, USA}
\email{solomyak@uw.edu}

\begin{abstract}
We study deviation of ergodic averages for dynamical systems
given by self-similar tilings on the plane and in higher dimensions.
The main object of our paper
is a special family of
finitely-additive measures for our systems.
An asymptotic formula is given for ergodic integrals in terms of these finitely-additive measures, and, as a corollary,
limit theorems are  obtained for dynamical systems given by self-similar tilings.
\end{abstract}

\date{\today}

\thanks{
A. B. is an Alfred P. Sloan Fellow and a Dynasty Foundation
Fellow.  He is supported in part by the Grant MK 4893.2010.1 of the President of
the Russian Federation, by the Programme on Dynamical Systems and Mathematical Control Theory
of the Presidium of the Russian Academy of Sciences, by RFBR-CNRS grant
10-01-93115 and by the RFBR grant 11-01-00654. \\
\indent B. S.  is supported in part by NSF grant DMS-0968879.
}

\maketitle

\thispagestyle{empty}

\section{Introduction}

We study the deviation of ergodic averages for certain tiling dynamical systems, namely,
translation $\R^d$-actions associated with self-similar tilings.
For $d=1$ asymptotic formulas and limit laws for such deviations were obtained in \cite{Bufetov}.
The main novelty of the $d\ge 2$ case is the appearance of ``boundary effects'' which result in some new phenomena.

We assume that the tilings have translationally finite local complexity, are aperiodic, and repetitive. 
Self-similarity means that there is an
expanding similarity map $\phi:\,\R^d\to \R^d$, such that every ``inflated'' tile can be subdivided into tiles  of the
original tiling, basically, a Markov property.
Given a self-similar tiling, we consider its orbit under translations and its closure in the natural ``local'' topology.
This is a compact metric space, on which $\R^d$ acts by translations.
The resulting dynamical system is known to be minimal and uniquely ergodic.
See the next section for precise definitions and statements.

Let $\Sf$ be the substitution matrix, which is primitive, and let $\theta_1,\ldots,\theta_m$ be its eigenvalues, ordered
by their absolute values:
$\theta_1 > |\theta_2|\ge \cdots \ge |\theta_m|$.
In order to describe our results, we make a simplifying assumption that there are no Jordan blocks associated with 
eigenvalues of absolute value $|\theta_2|$.
The most basic class of functions for which we consider the deviation of ergodic averages is the collection of
characteristic functions for ``cylinder sets'' of tiles. Averaging can be done over balls or cubes of diameter $R$, 
or more general increasing families of Lipschitz domains. A question arises: how can we estimate from above 
the deviation of the ergodic average from the mean? The answer depends on the relation between
$|\theta_2|$ and $\theta_1^{\frac{d-1}{d}}$, see Corollary~\ref{th-main1}. If $|\theta_2|<\theta_1^{\frac{d-1}{d}}$, 
then the deviation term is bounded above by $CR^{d-1}$, which means that the
main contribution comes from the boundary of the domain. On the other hand, if
$|\theta_2|> \theta_1^{\frac{d-1}{d}}$, 
then the  deviation term is bounded above by $CR^\alpha$, where $\alpha = \frac{d\log|\theta_2|}{\log\theta_1}\in (d-1,d)$
(if $|\theta_2|=\theta_1^{\frac{d-1}{d}}$, then there is a logarithmic correction). These deviation bounds are sharp, at least,
in the general case.
There are related recent results by Solomon
\cite{Solomon1,Solomon2} and Aliste-Prieto, Coronel, Gambaudo \cite{ACG,ACG2}, who obtained estimates 
for the rate of convergence to frequency of the number of prototiles per volume for a class of domains. They were
motivated by questions on bi-Lipschitz equivalence and bounded displacement of separated nets, arising
from self-similar tilings, to the lattice. The reader is referred to remarks at the end of Section 4
for a more detailed discussion of these results and how they compare to ours.

Our goal is a finer analysis of the deviation from the ergodic average, which we can perform in the
case $|\theta_2|> \theta_1^{\frac{d-1}{d}}$. 
The main tool here is a family of {\em finitely-additive measures} associated with the system.
It is known that the right and left eigenvectors of $\Sf$ corresponding to the dominant eigenvalue $\theta_1$ give rise
to the unique invariant probability
measure for the tiling dynamical system. The tiling space is locally a product of the ``Euclidean
leaf'' --- an open set in $\R^d$ --- and the transversal, which is a Cantor set with the structure
of a topological Markov chain. The invariant measure is
locally the product of the Lebesgue measure on $\R^d$ and a Markov measure. It turns out that for each eigenvalue $\theta$ of
$\Sf$, that is larger than $\theta_1^{\frac{d-1}{d}}$ in absolute value, 
one can associate two finitely-additive complex (or real signed) measures: one defined on an algebra of sets in $\R^d$ including
Lipschitz domains, and another one defined on the transversal. The latter one yields an invariant finitely-additive measure for the 
dynamical system, if we take a product (locally) with the Lebesgue measure.

Tilings can be viewed as multi-dimensional analogues of substitution
dynamical systems. By the Vershik-Livshits theorem \cite{Vershik,VerLiv}, primitive substitution
dynamical systems can be equivalently realized as Vershik's automorphisms
corresponding to Bratteli diagrams. Upper bounds for the deviation of
ergodic averages for substitution dynamical systems have been obtained
by Adamczewski \cite{Adamcz}; in the related context of interval exchange
transformations and translation flows on flat surfaces, such upper bounds
are due to Zorich \cite{Zorich} and Forni \cite{Forni}. An asymptotic formula for ergodic integrals
for translation flows has been obtained in \cite{Bufetov1},
relying on the construction of a special family of
finitely-additive invariant measures.
In particular, G. Forni's invariant distributions  are expressed through
the finitely-additive measures. Limit theorems for translation flows
follow as a corollary of the asymptotic formula.
We should mention that related objects (minimal cocycles with a scaling property)
for 1-dimensional symbolic substitutions have been studied by Kamae and
collaborators \cite{DKT,Kamae}.

As we said above, the main difficulty of the multi-dimensional case is due
to the more complicated behavior at the boundary.
In the one-dimensional case, finitely-additive measures are directly
constructed on ``Markovian'' arcs and then extended to general arcs by
exhaustion. In the multi-dimensional case, finitely-additive measures are
first constructed on tiles, and then the question arises of their
extension to rectangles, discs and so forth. Note, however, that while the
boundary of an interval consists of two points, the boundary of a
rectangle consists of several arcs, and their contribution need not be
negligible! 

Our first main result (see Theorem~\ref{th-main2}) is an asymptotic formula for the
deviation of the ergodic average in terms of the finitely-additive measures up to an error term, generically of order $R^{d-1}$.
Under the additional assumptions that the tiles are polyhedral,
 the similarity map $\phi$ is a pure dilation, and the second eigenvalue $\theta_2$ is real, simple, and satisfies
 $\theta_2> |\theta_3|$, we prove that
the deviations of ergodic averages obey a limit law: more precisely,
averages on cubes of side $r\lam^n$, appropriately normalized, converge in distribution to a non-degenerate random variable
(see Theorem~\ref{th-limitlaw}).

\section{Preliminaries}

We begin with tiling preliminaries, following \cite{SolJap}, see also \cite{LMS2,Robi,Sadun-book}.
We emphasize that our tilings are {\em translationally
finite}, thus excluding the pinwheel tiling \cite{radin} and its relatives.

\subsection{Tilings.}
Fix a set of types (or colors) labeled by $\{1,\ldots,m\}$.
A {\em tile} in $\R^d$ is defined as a pair $T = (A,i)$ where
$A  = \supp(T)$ (the support of $T$) is a compact set in $\R^d$ which is the
closure of its interior, and $i = \ell(T)\in \{1,\ldots,m\}$ is the type of $T$.
(The tiles are not assumed to
be homeomorphic to the ball or even connected. They may have fractal boundary.)
A {\em tiling} of $\R^d$ is a set $\Tk$ of tiles such that $\R^d = \bigcup
\{\supp(T):\ T\in \Tk\}$ and distinct tiles (or rather, their supports)
have disjoint interiors.

A patch $P$ is a finite set of tiles with disjoint interiors.
The {\em support of a patch} $P$ is defined by
$\supp (P) = \bigcup\{\supp(T):\ T\in P\}$. The {\em diameter of a patch}
$P$ is $\diam(P)=\diam(\supp(P))$. The {\em translate} of a tile $T=(A,i)$
by a vector $y\in \R^d$ is $T+y = (A+y, i)$. The translate
of a patch $P$ is $P+y = \{T+y:\ T\in P\}$. We say that two patches
$P_1,P_2$ are {\em translationally equivalent} if $P_2 = P_1+y$ for some
$g\in \R^d$. Finite subsets of $\Tk$ are called $\Tk$-patches.
For a set $\Om\subset \R^d$ we denote by
$$
\Tk|_\Om= \cup \{T\in \Tk:\,\supp(T)\subset \Om\}
$$
the patch of $\Tk$-tiles whose supports are contained in $\Om$.

\begin{defi} \label{def-flc}
A tiling $\Tk$ has (translational) {\em finite local complexity} (FLC) if
for any $R>0$ there are finitely many $\Tk$-patches of diameter less
than $R$ up to translation equivalence.
\end{defi}

\begin{defi} \label{def-rep}
A tiling $\Tk$ is called {\em repetitive} if for any patch $P\subset \Tk$
there is $R>0$ such that for any $x\in \R^d$ there is a
$\Tk$-patch $P'$ such that $\supp(P')\subset B_R(x)$ and $P'$ is a
translate of $P$.
%The minimal such $R$, denoted $R(P)$, is called the {\em repetitivity radius}
%of $P$.
\end{defi}

\subsection{Tile-substitutions, self-affine tilings.}
We study {\em perfect}
(geometric) substitutions, in which a tile is ``blown up'' by an expanding
linear map and then subdivided. 
A linear map $\phi : \R^d \rightarrow \R^d$ is {\em expansive}
if all its eigenvalues lie outside the unit circle.

\begin{defi}\label{def-subst}
Let $\Ak = \{T_1,\ldots,T_m\}$ be a finite set of tiles in $\R^d$
such that $T_i=(A_i,i)$; we will call them {\em prototiles}.
Denote by $\Pk_{\Ak}$ the set of
patches made of tiles each of which is a translate of one of $T_i$'s.
A map $\omega: \Ak \to \Pk_{\Ak}$ is called a {\em tile-substitution} with
expansion $\phi$ if
\begin{equation} \label{def-sub}
\supp(\om(T_j)) = \phi A_j \ \ \  \mbox{for} \  j\le m.
\end{equation}

In plain language, every expanded prototile $\phi T_j$ can be decomposed into
a union of tiles (which are all translates of the prototiles) with disjoint
interiors.
\end{defi}

The substitution $\om$ is extended to all translates of prototiles by
$\om(y+T_j)= \phi y + \om(T_j)$, and to patches by
$\om(P)=\cup\{\om(T):\ T\in P\}$. This is well-defined due to (\ref{def-sub}).
The substitution $\om$ also acts on the space of tilings whose tiles are
translates of those in $\Ak$.

To the substitution $\om$ we associate its $m \times m$
substitution matrix ${\sf S}$, with ${\sf S}_{ij}$
being the number of tiles of type $i$
in the patch $\om(T_j)$.
The substitution $\om$ is called {\em primitive}
if the substitution matrix is primitive, that is, if there exists $k\in \Nat$
such that ${\sf S}^k$ has only positive entries.
%We say that $\Tk$ is a fixed point of a substitution if $\om(\Tk) = \Tk$.

\begin{defi} \label{def-tsp}
Given a primitive tile-substitution $\om$, let
$X_\om$ be the set of all tilings whose
every patch is a translate of a subpatch of
$\om^n(T_j)$ for some $j\le m$ and $n\in \Nat$. (Of course, one can use
a specific $j$ by primitivity.) The set $X_\om$ is called the
{\em tiling space} corresponding to the substitution.
\end{defi}

\begin{defi} \label{def-saf}
A repetitive tiling $\Tk$, such that $\om(\Tk)=\Tk$ for a primitive tile-substitution $\om$,
is called a {\em self-affine tiling}.
The self-affine tiling is {\em self-similar} if the expansion map of $\om$ is a similitude, that is, for some $\lam >1$ we have
$$|\phi(x)| = \lam |x|,\ \ \mbox{for all}\ x\in \R^d.$$
The number $\lam$ is called the {\em real expansion constant}, or {\em linear dilatation}, of the map $\phi$.
\end{defi}

We say that a tile-substitution $\om$ has FLC
if for any $R>0$ there are finitely many subpatches of $\om^n(T_j)$
for all $j\le m$, $n\in \Nat$, of diameter less
than $R$, up to translation. This obviously implies that all tilings in
$X_\om$ have FLC, and is equivalent to it if the tile-substitution is
primitive.

\medskip

\noindent
{\bf Remark.}
A primitive substitution tiling
space is not necessarily of finite local complexity,
see \cite[p.244]{Ken1} and \cite{Danzer}. Thus we have
to assume FLC explicitly. Recently, tiling systems without FLC were studied in \cite{FraRob,FraSad}.

\begin{lemma}\cite[Prop.\,1.2]{prag} \label{lem-prag}
Let $\om$ be a primitive tile-substitution
of finite local complexity.
Then every tiling $\Sk\in X_\om$ is repetitive.
\end{lemma}

\subsection{Tile boundaries} For a tiling $\Tk$ denote by 
$$
\partial \Tk = \bigcup_{T\in \Tk} \partial (\supp(T))
$$
the union of the boundaries of all tile supports.

By the definition of a tiling, $\partial \Tk$ is nowhere dense in $\R^d$. For self-affine tilings, the boundary has zero Lebesgue measure.

\begin{lemma}\cite[Prop.\,1.1]{prag} \label{lem-bound1}
Let $\Tk$ be a self-affine tiling of $\R^d$. Then $\Lk^d(\partial \Tk)=0$.
\end{lemma}

There is also a kind of ``geometric rigidity'': if $\Tk$ is self-affine and
$\partial \Tk$ is piecewise smooth (even piecewise Lipschitz), then it has to be polyhedral.
This follows from the fact that $\partial \Tk$ is invariant under the expanding linear map.

\subsection{Tiling topology and tiling dynamical system}
We use a tiling metric on $X_\om$, which is based on a
simple idea: two tilings are close if after a small translation they agree
on a large ball around the origin. There is more than one way
to make this precise, and our formal definition is as follows:
For $\Tk_1,\Tk_2 \in X_\om$ let
$$
\widetilde{d}(\Tk_1,\Tk_2) := \inf\{r \in (0,2^{-1/2}):
\ \exists\,y,\ |y| \le r,\ \ 
\supp((\Tk_1-y) \cap \Tk_2) \supset B_{1/r}(0)\}.
$$
%$$
%\Tk_1-g \mbox{ \ agrees with $\Tk_2$ on $B_{1/r}(0)$} \}.
%$$
Then
$
d(\Tk_1,\Tk_2) = \min\{2^{-1/2},\widetilde{d}(\Tk_1,\Tk_2)\}
$
is a metric on $X_\om$.

\begin{theorem}\cite{Rud} {\em (see also \cite{Robi}).}
$(X_\om,d)$ is a complete metric space.
It is compact, whenever the space has finite local complexity.
The action of $\R^d$ by translations on
$X_\om$, given by  $\Sk\mapsto \Sk-y$, $y\in \R^d$, is continuous.
\end{theorem}

This continuous translation action $(X_\om,\R^d)$ is called the
(topological) tiling dynamical system associated with the tile-substitution.

\begin{theorem} \label{th-min}
If $\om$ is a primitive tiling substitution with FLC, then
the dynamical system $(X_\om,\R^d)$ is minimal, that is, for every
$\Sk\in X_\om$, the orbit $\{\Sk-y:\ y\in \R^d\}$ is dense in $X_\om$.
\end{theorem}

This follows from Lemma~\ref{lem-prag} and Gottschalk's Theorem
\cite{Gott}, see \cite[Sec.\,5]{Robi} for details.

Recall that a topological dynamical system is said to be {\em uniquely ergodic} if it has
a unique invariant Borel probability measure.

\begin{theorem} \label{unerg}
If $\om$ is a primitive tiling substitution with FLC, then
the dynamical system $(X_\om,\R^d)$ is uniquely ergodic.
\end{theorem}

This result has appeared in the literature in several
slightly different versions.
We refer to \cite[Theorem 4.1]{LMS2} and \cite{Robi} for the proof.

Let $\mu$ be the unique invariant measure from Theorem~\ref{unerg}.
The measure-preserving tiling dynamical system is denoted by
$(X_\om,\R^d,\mu)$.

\begin{lemma} {\em (see \cite[Th.\,5.10]{Robi})} If $\om$ is a primitive tiling substitution with FLC, then there exists $k\in \N$ and
$\Tk\in X_\om$ such that $\om^k(\Tk)=\Tk$.
\end{lemma}

Combining this with Lemma~\ref{lem-bound1}, we obtain that $\Lk^d(\partial \Sk)=0$ for all $\Sk\in X_\om$.

\subsection{Substitution action.}
The substitution $\om$ acts on the entire space $X_\om$, and it is easy to see from the definition of $X_\om$ that
$\om:\, X_\om \to X_\om$ is surjective. 
We will address the question of its invertibility, but first record the obvious relation:
\be \label{eq-infl}
\om(\Tk-y) = \om(\Tk) - \phi y.
\ee

\begin{lemma} \label{lem-invmeas}
If $\om$ is a primitive substitution with $FLC$, and $\mu$ is the unique invariant probability measure for the translation action $(X_\om,\R^d)$, then
$\mu$ is invariant under the substitution action $\om$ (in general, non-invertible). 
\end{lemma}

\begin{proof}
Consider the probability measure $\om_*\mu = \mu\circ \om^{-1}$ on $X_\om$. It is immediate from (\ref{eq-infl}) that this measure is invariant under
the translation action, hence $\mu=\om_*\mu$ by unique ergodicity.
\end{proof}

\begin{defi}
A primitive tile-substitution $\om$ is called {\em non-periodic} if all $\Tk\in X_\om$ are non-periodic, that is,
$\Tk-y = \Tk$ implies $y=0$.
If at least one $\Tk\in X_\om$ is non-periodic, then $\om$ is non-periodic by minimality.
\end{defi}

\begin{theorem}[\cite{Sol-ucp}] \label{th-ucp} The map $\om:\, X_\om \to X_\om$ is injective if and only if $\om$ is
a non-periodic substitution.
\end{theorem}

We assume that $\om$ is non-periodic for the rest of the paper.

For non-periodic substitutions we have the $\Z$-action generated by $\om$, along with the translation $\R^d$-action.
It is useful to note that this $\Z$-action is, in some sense, {\em hyperbolic}, with the
orbit of $\Tk$ under the $\R^d$-action playing the role of the unstable set of $\Tk$ (clear from (\ref{eq-infl})), and the
the {\em transversal} containing $\Tk$ playing the role of the stable set, see
\cite{AP} for details. (The transversal is defined here as the set of all tilings
which agree with $\Tk$ exactly on some patch, possibly one tile, containing the origin in its interior; see Section 4 for precise
definitions.)

The substitution $\om$ can be extended to ``super-tiles'' and ``sub-tiles'' of all orders. More precisely, for any $k\in \Z$ let
$$
\phi^k \Ak:= \{\phi^k(T_i)\}_{i=1}^m,\ \ (\phi^k \om) (\phi^k T_i) = \phi^k(\om(T_i)).
$$
This defines a bi-infinite sequence of tile-substitutions and the corresponding tiling spaces $\{X_{\phi^k\om}\}_{k\in \Z}$.
The {\em subdivision map}
\be \label{eq-subdiv}
\Upsilon_k:\ X_{\phi^k \om} \to X_{\phi^{k-1}\om}
\ee
acts by subdividing each tile according to the rule implicit in the substitution. The inverse of the subdivision map is the
{\em composition map}, which is well-defined if and only if the substitution is non-periodic, by Theorem~\ref{th-ucp}. For $\Tk \in X_\om$
denote $\Tk^{(0)}:= \Tk$, and
$$
\Tk^{(k)} = \left\{\begin{array}{ll} \Upsilon_k^{-1}\cdots \Upsilon_1^{-1}(\Tk), & \mbox{if}\ k> 0,\\
\Upsilon_{k+1}\cdots \Upsilon_0(\Tk), & \mbox{if}\ k < 0;\end{array} \right.\ \ \Tk^{(k)}\in X_{\phi^k\om}.
$$
Note that $\Tk^{(k)}$ is always defined for $k<0$, and if $\om$ is non-periodic, then for $k>0$ as well.

\subsection{Some geometric measure theory}

Denote by $\Hk^\alpha$ the $\alpha$-dimensional Hausdorff measure (see e.g. \cite{mattila} for definitions and
basic properties) and by $\Lk^d$ the Lebesgue measure in $\R^d$.
A set $H\subset \R^d$ is said to be $m$-rectifiable for $m\in \N,\ m<d$, if $\Hk^m(H)>0$ and
there exist Lipschitz maps $h_i:\,\R^m\to \R^d$,\ $i\in \N$, 
such  that
\be \label{def-rectif}
\Hk^m\left(H\setminus \bigcup_{i=1}^\infty h_i(\R^m)\right)=0.
\ee
See e.g.\ \cite[p.204]{mattila}.

We will say that an open {\em bounded} set $\Om\subset \R^d$ is a {\em Lipschitz domain} if there exist finitely many Lipschitz maps $h_i:\,
\R^{d-1}\to \R^d$, $i\le N$, such that
\be \label{def-Lips}
\Hk^{d-1}\left(\partial \Om \setminus \bigcup_{i=1}^\infty h_i(\R^{d-1})\right)=0.
\ee
Thus, the boundary of a Lipschitz domain is $(d-1)$-rectifiable.
For $A\subset \R^d$ denote 
$$
U(A,r) = \{x\in \R^d:\ \dist(x,A) \le r\}.
$$
The $\alpha$-dimensional upper Minkowski content of $A$ is defined by 
$$
\Mk^{*\alpha}(A) = \limsup_{r\to 0} (2r)^{d-\alpha} \Lk^d(U(A,r)).
$$
It is known \cite{federer} that $\Mk^{*m}(A) = \Hk^m(A)$ for an $m$-rectifiable set $A$. It follows that 
for any $(d-1)$-rectifiable set $A$ and $b>0$ there exists $C(A,b)$ such that
\be \label{eq-gmt}
\Lk^d(U(A,r)) \le C(A,b) r,\ \ \mbox{for all}\ r\in (0,b].
\ee
Indeed, we have $\Mk^{*(d-1)}(A)>0$, hence (\ref{eq-gmt}) holds for some $b=b_0$, and then we can simply
take 
$$
C(A,b) = \max\left\{C(A,b_0),\frac{\Lk^d(U(A,b))}{b_0\Lk^d(U(A,b_0))}\right\}\ \ \mbox{for}\ \  b>b_0.
$$

\subsection{Linear algebra.} We will need the following well-known result.

\begin{lemma} \label{lem-jordan}
Let $J$ be a Jordan cell of size $s$ with diagonal entries $\theta$, $|\theta|>1$. Then
$$
\|J^k\| \le \left\{\begin{array}{ll} sk^{s-1}|\theta|^k & \mbox{for}\ k>0; \\
                                     s|k|^{s-1}|\theta|^{k+s} & \mbox{for}\ k<0, \end{array} \right.
$$
where $\|\cdot\|$ is the operator matrix norm induced by the Euclidean norm.
\end{lemma}

%%%%%%%%%%%%%%%%%%%%%%%%%%%%%%%%%%%%%%%%%%%%%%%%%%%%%%%%%%%%%%%%%%%%%%%%%%%%%%%%%%%%%%%%%%%%%%%%%%

\section{Finitely-additive measures on Lipschitz domains.}
Recall that we have the prototiles $T_i,\ i=1,\ldots,m$, with $\supp(T_i)=A_i$. 
By the definition of self-similar
tiling and the substitution matrix, we have
\be \label{eq-dij}
A_j = \bigcup_{i=1}^m (\phi^{-1}A_i + \phi^{-1} D_{ij}),\ \ j=1,\ldots,m,
\ee
where the union is almost disjoint (up to the boundaries), and $D_{ij}$ are finite sets, with $\Sf_{ij} = \# D_{ij}$.
Then we obtain, using Lemma~\ref{lem-bound1}, that
$$
\sum_{i=1}^m {\sf S}_{ij} \Lk^d(A_i) = \Lk^d(\supp(\om(T_j))= \Lk^d(\lam A_j) = \lam^d\Lk^d(A_j),\ \ j=1,\ldots,m,
$$
where $\lam$ is the linear dilatation of $\phi$.
It follows that $(\Lk^d(A_j))_{j=1}^m$ is the Perron-Frobenius eigenvector for ${\sf S}^t$,
with the eigenvalue $\lam^d$, hence the dominant eigenvalue of $\Sf^t$ is 
\be \label{lambda}
\th_1 = \lam^d.
\ee

Denote by $E^+$ the linear span of the Jordan cells of the transpose of the
substitution matrix
$\Sf^t$ corresponding to eigenvalues greater than 1. Our goal is to define finitely-additive (complex) measures
$\Phi^+_{v,\Tk}$ for $v\in E^+$ and $\Tk\in X_\om$, analogous to those from \cite{Bufetov0,Bufetov,Bufetov1}, 
which corresponds to $d=1$.

Recall that for each $\Tk\in X_\om$ we have a sequence $\{\Tk^{(k)}\}_{k\in \Z}$ of ``sub-tilings''
(for $k<0$) and ``super-tilings'' (for $k>0$), together with $\Tk^{(0)}=\Tk$, such that $\Tk^{(k-1)}$ is obtained from $\Tk^{(k)}$ by the
process of subdivision. They are uniquely defined by the assumption of non-periodicity, see Subsection 2.5.

Initially, $\Phi^+_{v,\Tk}$ is defined on the following ring of
subsets of $\R^d$:
$$
\Ck^+_{\Tk}:= \bigcup_{k\in \Z} \{F=(\supp(P)\setminus N_1) \cup N_2 :\ P\subset \Tk^{(k)},\ N_1, N_2\subset \partial \Tk^{(k)}\}. 
$$
(Recall that a ring of sets is closed under finite unions and intersections, but not necessarily under complements.)
In other words, this is the ring generated by sub-tiles and super-tiles of $\Tk$ of all orders (referred to as ``tiles
of order $k$'' for $k\in \Z$) and arbitrary subsets of their boundaries.
The finitely-additive measure $\Phi^+_{v,\Tk}$ is defined on tiles of order $k\in \Z$ by
\be \label{eq-mdef1}
\Phi^+_{v,\Tk}(\supp(T)) = ((\Sf^t)^k v)_j,\ \mbox{if}\ \exists\, k\in \Z,\ y\in \R^d:\ T=\phi^k(T_j)-y\in \Tk^{(k)},
\ee
and
$$
\Phi^+_{v,\Tk}(N)=0\ \ \ \mbox{if}\ N\subset \partial \Tk^{(k)}.
$$
Then this set function is extended to finite disjoint unions by additivity.
We need to show that this definition is consistent; then we get a finitely-additive set function on the ring $\Ck^+_{\Tk}$.
Clearly, it suffices to verify finite additivity for a tile and its
decomposition into sub-tiles.
We have
$$
\Phi^+_{v,\Tk}(A_j-y)= v_j,\ \ \ \mbox{where}\ T_j - y \in \Tk,
$$
and hence (\ref{eq-dij}) implies
$$
\sum_{i=1}^m \sum_{x\in D_{ij}} \Phi^+_{v,\Tk}(\phi^{-1}A_i +x -y) = \sum_{i=1}^m (\Sf^t)_{ji} ((\Sf^t)^{-1} v)_i = v_j,
$$
as desired. 
Note that $(\Sf^t)^{-1} v$ is well-defined because $v$ is in the expanding subspace for $\Sf^t$.
We showed finite additivity when subdividing a $\Tk$-tile into $\Tk^{(-1)}$-tiles; the general case follows.

Observe that $v\mapsto \Phi^+_{v,\Tk}$ is linear, so we can restrict ourselves to $v$ from a basis of $E^+$; specifically, to a basis
of eigenvectors and root vectors of $\Sf^t$, associated to the canonical Jordan form of $\Sf^t$.

\begin{defi} \label{def-rapid}
Define the {\bf rapidly expanding subspace of $E^{++}$ of $\Sf^t$} to be the linear span of Jordan cells
of eigenvalues satisfying the inequality
$$
|\theta|> \theta_1^{\frac{d-1}{d}}=\lam^{d-1},
$$
where $\lam$ is the linear dilatation of $\phi$, see (\ref{lambda}).
\end{defi}

The space $E^{++}$ yields finitely-additive measures for which the main contribution is in the interior of the set, rather than at
the boundary. Heuristically, the main contribution of the eigenvalues with absolute values in $(1, \lam^{d-1})$
will, on the contrary, be concentrated at the boundary. At any rate, it will turn out that the contribution of the latter eigenvalues
to the deviation of the ergodic average on, say, a ball of radius $R$ is bounded above by $CR^{d-1}$, see (\ref{eq-rapid}) below.
This effect was not present in the one-dimensional case, since for $d=1$ we have $E^{++}=E^+$.

Let $\gam>0$ be the average of $\lam^{d-1}$ and the smallest absolute value of an eigenvalue of $\Sf^t|_{E^{++}}$.
Then $\gam>\lam^{d-1}$, and we have
\be \label{eq-jordan}
\|(\Sf^t)^kv\|\le \const\cdot \gam^k\|v\|,\ \ \ k<0,\ \ \forall\,v\in E^{++},
\ee
where the constant depends only on the matrix $\Sf$.

Denote by $\Qk$ the ring of sets generated by
Lipschitz domains in $\R^d$ and subsets of their boundaries for a fixed $d$ (recall that Lipschitz domains are assumed to be {\bf bounded} by
definition).
We will show that for $v\in E^{++}$ the finitely-additive measure $\Phi^+_{v,\Tk}$ 
can be defined in a natural way on $\Qk$.
If the tile supports belong to $\Qk$ (in which case they have to be polyhedral, see Subsection 2.3), then this is an extension of the
corresponding finitely-additive measure defined on $\Ck^+_\Tk$. But we also allow ``fractal'' boundaries, in which case it is not clear whether
the finitely-additive measure can be extended to the ring generated by $\Ck^+_\Tk \cup \Qk$.

\begin{lemma} \label{lem-finadd1}
For any $v\in E^{++}$,
there exist finitely-additive measures $\Phi^+_{v,\Tk}$ defined on the ring $\Qk$. 
Moreover, they satisfy the following ``cocycle''
conditions for any $\Om \in \Qk$:
\be \label{coc1}
\Phi^+_{v,\Tk-y}(\Om) = \Phi^+_{v,\Tk}(\Om+y)\ \ \mbox{for all}\ \Tk\in X_\om,\ y\in \R^d,
\ee
\be \label{coc2}
\Phi^+_{\Sf^t v,\Tk}(\Om) = \Phi^+_{v,\om(\Tk)}(\phi(\Om)).
\ee
In particular, for an eigenvector $v$, with $\Sf^t v=\theta v$, we get
\be \label{coc3}
\Phi^+_{v,\om(\Tk)}(\phi(\Om)) = \Phi^+_{\theta v,\Tk}(\Om)=\theta \Phi^+_{v,\Tk}(\Om).
\ee
\end{lemma}

Before the proof, we introduce a construction which will be used throughout the paper. It is an efficient hierarchical
``packing'' of a Lipschitz domain by tiles of varying orders. Analogous constructions have been used in \cite{Solomon1,Solomon2,ACG,ACG2,Sadun}.

Fix a tiling $\Tk\in X_\om$ and a Lipschitz domain $\Om$.
Recall that $\Tk^{(k)}|_\Om$ denotes the collection of $\Tk^{(k)}$-tiles whose supports lie in
$\Om$. Observe that $\supp(\Tk^{(k)}|_\Om) \in \Ck^+_\Tk$ and $\supp(\Tk^{(k+1)}|_\Om)\subset \supp(\Tk^{(k)}|_\Om)$. 
Further, let
\be \label{def-rk}
\Rk^{(k)}(\Om):= \{T\in \Tk^{(k)}|_\Om:\ \supp(T) \not\subset \supp(\Tk^{(k+1)}|_\Om)\}.
\ee
In words, $\Rk^{(k)}(\Om)$ consists of those tiles of order $k$ in $\Om$ which belong to tiles of
order $k+1$ that {\bf do not} lie in $\Om$, hence these tiles of order $k+1$ must intersect the boundary $\partial \Om$.
Let $d_{\max}, d_{\min}$ be the largest and the smallest diameter
of a $\Tk$-tile respectively. Since the largest diameter of a $\Tk^{(k+1)}$-tile equals $d_{\max} \lam^{k+1}$,
it follows that
$$
\supp(\Rk^{(k)}(\Om)) \subset U(\partial \Om, d_{\max} \lam^{k+1}).
$$
Denote by $a_{\min} = \min_{j\le m} \Lk^d(A_j)$ the smallest volume of a prototile.
Then the volume of an order $k$ tile is at least
$a_{\min}\lam^{dk}$, and we obtain
\be \label{eq-estrk}
\# \Rk^{(k)}(\Om) \le \Lk^d(U(\partial \Om, d_{\max} \lam^{k+1}))  a_{\min}^{-1} \lam^{-dk}.
\ee

\begin{proof}[Proof of Lemma~\ref{lem-finadd1}]
We define for a Lipschitz domain $\Om$:
\be \label{def-phiv}
\Phi^+_{v,\Tk}(\Om):= \lim_{k\to -\infty} \Phi^+_{v,\Tk}(\supp(\Tk^{(k)}|_\Om)).
\ee
Let us show that the limit exists (note that we cannot use monotonicity, since the values of $\Phi^+_{v,\Tk}$ need not be
positive or even real).
By (\ref{eq-estrk}) and (\ref{eq-gmt}),
$$
\# \Rk^{(k)}(\Om) \le C(\partial \Om,1) d_{\max}\lam  a_{\min}^{-1} \lam^{-(d-1)k} = \const\cdot \lam^{-(d-1)k}
$$
for $k\in \Z$ such that $d_{\max}\lam^{k+1}\le 1$, that is, for
$$
k \le -\log d_{\max}-\log\lam -1.
$$
Now, by finite additivity of $\Phi^+_{v,\Tk}$ on $\Ck^+_\Tk$, in view of (\ref{eq-mdef1}),
\begin{eqnarray*}
|\Phi^+_{v,\Tk}(\supp(\Tk^{(k)}|_\Om)) - \Phi^+_{v,\Tk}(\supp(\Tk^{(k+1)}|_\Om))| & = & |\Phi^+_{v,\Tk}(\supp(\Rk^{(k)}(\Om))| \\
               & \le & \sum_{T\in \Rk^{(k)}(\Om)} |\Phi^+_{v,\Tk}(\supp(T))| \\
               & \le & \# \Rk^{(k)}(\Om)\cdot \|(\Sf^t)^k v)\| \\
               & \le & \const\cdot \lam^{-(d-1)k} \gamma^k\|v\|,
\end{eqnarray*}
for $k<\min\{0, -\log d_{\max}-\log\lam -1\}$, where we used (\ref{eq-jordan}) in the last step.
By assumption, $|\gamma|>\lam^{d-1}$, 
hence the last expression tends to zero exponentially fast as $k\to -\infty$, and the existence of the
limit in (\ref{def-phiv}) is verified. We then define
$$
\Phi^+_{v,\Tk}(\Om\cup N) = \Phi^+_{v,\Tk}(\Om)\ \ \mbox{for}\ N\subset \partial \Om.
$$

Now let us check finite additivity. If $\Om_1$ and $\Om_2$ are Lipschitz domains with disjoint interiors such that
$$
\ov{\Om} = \ov{\Om}_1\cup \ov{\Om}_2
$$
for another Lipschitz domain $\Om$, then 
$$
\Tk^{(k)}|_\Om \setminus (\Tk^{(k)}_{\Om_1} \cup \Tk^{(k)}_{\Om_2}) \subset \Rk^{(k)}(\Om_1)\cap\Rk^{(k)}(\Om_2),
$$
since the former consists of those $\Tk^{(k)}$-tiles which intersect $\partial \Om_1\cap \partial \Om_2$.
The estimate above shows that the $\Phi^+_{v,\Tk}$-measure of the support of the latter patch
tends to zero as $k\to - \infty$. The definition (\ref{def-phiv})
then shows 
$$\Phi^+_{v,\Tk}(\Om) = \Phi^+_{v,\Tk}(\Om_1) + \Phi^+_{v,\Tk}(\Om_2).$$
Thus we can define $\Phi^+_{v,\Tk}$ on a finite union of disjoint Lipschitz domains and subsets of their boundaries
consistently, and finite additivity follows.

Formulas (\ref{coc1}) and (\ref{coc2}) are easily verified: 
they hold on the ring $\Ck_\Tk^+$ by definition, hence they hold for Lipschitz domains by (\ref{def-phiv}), and
therefore for all the elements of the ring $\Qk$.
\end{proof}

\smallskip

In the next lemma, we estimate the growth of our finitely-additive measures on dilations of a given Lipschitz domain.

We will often use constants which depend only on the tiling substitution $\om$ 
(which includes all the data for the tiling space $X_\om$), and on the
domain $\Omega$. This will be often written as $C=C(\om,\Om)$ for short.

\begin{lemma} \label{lem-finadd2}
Suppose that $v\in E^{++}$, with $\|v\|=1$, belongs to the $\Sf^t$-invariant
subspace corresponding to a Jordan block of
size $s\ge 1$, with an eigenvalue $\theta$ ($v$ is an eigenvector if $s=1$). Then 
for a Lipschitz domain $\Om$ and $\Om_R=R\Om$, we have for $R\ge 2$:
\be \label{eq-hoelder}
|\Phi^+_{v,\Tk}(\Om_R)|\le C_1 (\log R)^{s-1} R^\alpha,\ \ \mbox{where}\ \alpha = \frac{\log|\theta|}{\log \lam}\,,
\ee
for a constant $C_1=C_1(\om,\Om)$. 
\end{lemma}

\begin{proof} 
Fix a tiling $\Tk\in X_\om$.
For $k\in \Z$ we consider patches $\Rk^{(k)}(\Om)$ introduced in (\ref{def-rk}).
If we fix $\Om$ and let $\Om_R = R\Om$, then we obtain by (\ref{eq-estrk}),
\begin{eqnarray}
\# \Rk^{(k)}(\Om_R) & \le & \Lk^d(U(\partial \Om_R, d_{\max} \lam^{k+1})) a_{\min}^{-1} \lam^{-dk} \nonumber \\
                    &  =  & R^d\Lk^d(U(\partial \Om, d_{\max} \lam^{k+1}R^{-1})) a_{\min}^{-1} \lam^{-dk}. \label{103.2}
\end{eqnarray}
Denote
$$
k_R = \max\{k\in \Z:\ \Rk^{(k)}(\Om_R) \ne \es\}.
$$
It is clear that $d_{\min}\lam^{k_R}\le \diam(\Om_R)=R\,\diam(\Om)$, so
\be \label{eq-kr}
\lam^{k_R}\le R\,\diam(\Om)/d_{\min}.
\ee
Thus, in (\ref{103.2}) we have $d_{\max} \lam^{k+1}R^{-1} \le  \lam\,\diam(\Om)d_{\max}/d_{\min}:=b$, hence
\be \label{103}
\# \Rk^{(k)}(\Om_R) \le C'\cdot R^{d-1} \lam^{-(d-1)k},\ \ \ \mbox{with}\ \ 
C' = C(\partial \Om,b)d_{\max}\lam a_{\min}^{-1}.
\ee
where we used (\ref{eq-gmt}). 
Note that the constant $b$, and hence $C'$, depends only on $\om$ and $\Om$.
We have by (\ref{def-phiv}):
\be \label{104}
\Phi^+_{v,\Tk}(\Om_R) = \sum_{k=-\infty}^{k_R}\sum_{T\in \Rk^{(k)}(\Om_R)} \Phi^+_{v,\Tk}(\supp(T)).
\ee
Therefore, 
\begin{eqnarray}
|\Phi^+_{v,\Tk}(\Om_R)| & \le & C'\cdot R^{d-1} \sum_{k=-\infty}^{k_R} \frac{\|(\Sf^t)^k v)\|}{\lam^{(d-1)k}} \label{104.1}\\
                        & \le & C''\cdot R^{d-1}\left( \sum_{k=-\infty}^{-1} \frac{|k|^{s-1}|\theta|^{k+s}}{\lam^{(d-1)k}} + 1 + 
   \sum_{k=1}^{k_R} \frac{|k|^{s-1}|\theta|^{k}}{\lam^{(d-1)k}}\right) \label{104.2} \\
                       & \le & C'''\cdot R^{d-1} \frac{|k_R|^{s-1} |\theta|^{k_R}}{\lam^{(d-1)k_R}}. \label{105}
\end{eqnarray}
We used (\ref{103}) and (\ref{eq-mdef1}) in (\ref{104.1}), Lemma~\ref{lem-jordan} in (\ref{104.2}), and the assumption $|\th|>\lam^{d-1}$
in (\ref{105}). Note that the constants $C'', C'''$ depend only on $\om$ and $\Om$ (they depend on the substitution
matrix, which is encoded in $\om$).
We also assumed that $k_R>0$, 
but this does not lead to 
loss of generality since it is enough to establish (\ref{eq-hoelder}) for
$R$ sufficiently large.
In view of (\ref{eq-kr}), we have
$$
\frac{|\theta|^{k_R}}{\lam^{(d-1)k_R}}\le \const\cdot R^{\frac{\log(|\theta|/\lam^{d-1})}{\log \lam}} =
\const\cdot R^{\frac{\log|\theta|}{\log\lam}-(d-1)},
$$
hence the inequality (\ref{105}) implies
$$
|\Phi^+_{v,\Tk}(\Om_R)| \le C_1 (\log R)^{s-1} R^{\log|\theta|/\log\lam},
$$
with $C_1=C_1(\om,\Om)$, as desired.
\end{proof}

We record the following fact, which easily follows from the proof of Lemma~\ref{lem-finadd2}, for future use.
The notation $\Tk|_{\Om_R}$, used in the lemma below, means the collection of all $\Tk$ tiles contained in $\Om_R$.

\begin{lemma} \label{lem-finadd2.cor}
For a Lipschitz domain $\Om$ there exists a constant $C_2=C_2(\om,\Om)>0$
such that for all $v\in E^{++}$, with $\|v\|=1$, and for all $\Tk\in X_\om$, we have for $\Om_R=R\Om$:
\be \label{corest}
|\Phi^+_{v,\Tk}(\Om_R) - \Phi^+_{v,\Tk}(\supp(\Tk|_{\Om_R})| \le C_2 R^{d-1},\ \ \mbox{for all}\ R\ge 1.
\ee
\end{lemma}

\begin{proof}
We have
$$
\Phi^+_{v,\Tk}(\supp(\Tk|_{\Om_R})) = \sum_{k=0}^{k_R}\sum_{T\in \Rk^{(k)}(\Om_R)} \Phi^+_{v,\Tk}(\supp(T)).
$$
Comparing with (\ref{104}) and using (\ref{103}) we obtain, similarly to (\ref{105}):
\begin{eqnarray*}
|\Phi^+_{v,\Tk}(\Om_R) - \Phi^+_{v,\Tk}(\supp(\Tk|_{\Om_R})| & \le &
\left|\sum_{k=-\infty}^{-1}\sum_{T\in \Rk^{(k)}(\Om_R)} \Phi^+_{v,\Tk}(\supp(T))\right| \\
& \le & C'R^{d-1} \sum_{k=-\infty}^{-1} \frac{\|(\Sf^t)^k v)\|}{\lam^{(d-1)k}} \\
                       & \le & C'R^{d-1} \cdot C'' \sum_{k=-\infty}^{-1} \frac{\gamma^{k_R}}{\lam^{(d-1)k_R}} \\
                       & \le & C_2 R^{d-1},
\end{eqnarray*}
where $\gamma>\lam^{d-1}$ is from (\ref{eq-jordan}).
\end{proof}

\subsection{H\"older estimates.}
Next we establish a H\"older estimate for our finitely-additive measures. Although more general
Lipschitz domains could be handled, we restrict ourselves to cubes, for simplicity and because the limit law in Section~\ref{section_limitlaw}
below is obtained in this setting. We do not need this result until Section~\ref{section_limitlaw}.
Denote
$$
Q_r:= [-r/2,r/2]^d\ \ \ \mbox{and}\ \ \ A_{r_1,r_2}:= Q_{r_2}\setminus \Int(Q_{r_1})\ \ \mbox{for}\ 0\le r_1 < r_2.
$$
Thus, $A_{r_1,r_2}$ is the closed ``annulus'' between two concentric cubes.

\begin{lemma} \label{lem-finadd3}
Suppose that $v\in E^{++}$, with $\|v\|=1$,
satisfies $\Sf^tv = \theta v$ (so that $\theta>\lam^{d-1}$ by the definition of the rapidly expanding subspace $E^{++}$). 
Then there exists a constant $C_3=C_3(\om,\Om)>0$
such that for any $\Tk\in X_\om$ and any $0\le r_1<r_2$ we have
\be \label{eq-hoelder2}
|\Phi^+_{v,\Tk}(Q_{r_2}) - \Phi^+_{v,\Tk}(Q_{r_1})| \le C_3  r_2^{d-1} (r_2-r_1)^{\alpha-(d-1)},\ \ 
\mbox{where}\ \alpha = \frac{\log|\theta|}{\log \lam}\,.
\ee
\end{lemma}

\noindent {\bf Remark.} Taking $r_1=0$ we obtain the upper bound $C_3 r_2^\alpha$, which agrees with (\ref{eq-hoelder}), since
$s=1$ for the eigenvector $v$.

\begin{proof}
We have
by finite additivity:
$$
\Phi^+_{v,\Tk}(Q_{r_2}) - \Phi^+_{v,\Tk}(Q_{r_1})=\Phi^+_{v,\Tk}(A_{r_1,r_2}).
$$
Consider $\Rk^{(k)}(A_{r_1,r_2})$ as defined in (\ref{def-rk}); recall that
$$
\# \Rk^{(k)}(A_{r_1,r_2}) \le \Lk^d(U(\partial A_{r_1,r_2}, d_{\max} \lam^{k+1}))a_{\min}^{-1}\lam^{-dk}
$$
by (\ref{eq-estrk}).
Clearly, $\partial A_{r_1,r_2}=\partial Q_{r_1} \cup \partial Q_{r_2}$. The following claim is elementary.

\medskip

\noindent {\bf Claim.} {\em For any $r>0$ and $t\in (0,r)$,
\be \label{eq-claim}
\Lk^d(U(\partial Q_r,t/2))< d\, 2^d t r^{d-1}.
\ee
}
Indeed, we have $\Lk^d(U(\partial Q_r,t/2))< (r+t)^d-(r-t)^d$ whence (\ref{eq-claim}) follows
by a simple calculus exercise.

\smallskip

Therefore, for all $0\le r_1 < r_2$,
$$
\Lk^d(U(\partial A_{r_1,r_2}, t/2)) \le \Lk^d(U(\partial Q_{r_1},t/2)) + \Lk^d(U(\partial Q_{r_2},t/2))\le d\,2^{d+1} t r_2^{d-1}.
$$
Thus,
\be \label{108}
\# \Rk^{(k)}(A_{r_1,r_2}) \le C(d,\om) \lam^{-(d-1)k}r_2^{d-1},\ \mbox{where}\ \ C(d,\om) = d\, 2^{d+2} d_{\max} \lam a_{\min}^{-1}.
\ee
Let $k_0=\max\{k\in \Z:\ \Rk^{(k)}(A_{r_1,r_2}) \ne \es\}$. 
We have
$$
\Phi^+_{v,\Tk}(A_{r_1,r_2})=\sum_{k=-\infty}^{k_0}\ \sum_{T\in \Rk^{(k)}(A_{r_1,r_2})} \Phi^+_{v,\Tk}(\supp(T)).
$$
Then we obtain, using (\ref{108}),
\begin{eqnarray} 
|\Phi^+_{v,\Tk}(A_{r_1,r_2})| & \le & C(d,\om) r_2^{d-1} \sum_{k=-\infty}^{k_0} \frac{\|(\Sf^t)^k v\|}{\lam^{(d-1)k}} \nonumber \\
                              & =   &  C(d,\om) r_2^{d-1} \sum_{k=-\infty}^{k_0} \frac{|\theta|^k}{\lam^{(d-1)k}}  \nonumber \\
                              & =   & \frac{C(d,\om)r_2^{d-1}}{1-\frac{\lam^{d-1}}{|\theta|}} \left( \frac{|\theta|}{\lam^{d-1}}\right)^{k_0}.
\label{dumb1}
\end{eqnarray}
It remains to estimate $k_0$. By definition, a tile of order $k_0$ must be contained in $A_{r_1,r_2}$. Let $\eta>0$ be
such that every $\Tk$ prototile contains a ball of radius $\eta$ in its interior. Then a ball of radius $\lam^{k_0}\eta$
must be contained in $A_{r_1,r_2}$. It is easy to see (we do not attempt to get a sharp estimate here) that 
$$
\lam^{k_0}\eta \le r_2-r_1,
$$
since the center of the ball must have at least one coordinate in the interval $[r_1,r_2]$.
It follows that 
$$
\left( \frac{|\theta|}{\lam^{d-1}}\right)^{k_0} \le \left(\frac{r_2-r_1}{\eta}\right)^{\frac{\log|\theta|}{\log\lam}-(d-1)},
$$ 
whence (\ref{dumb1}) yields
\be \label{dumb2}
|\Phi^+_{v,\Tk}(A_{r_1,r_2})| \le C_3 r_2^{d-1}(r_2-r_1)^{\frac{\log|\theta|}{\log\lam}-(d-1)},
\ee
with the constant $C_3$ depending only on the tiling substitution $\om$, as desired.
\end{proof}

%%%%%%%%%%%%%%%%%%%%%%%%%%%%%%%%%%%%%%%%%%%%%%%%%%%%%%%%%%%%%%%%%%%%%%%%%%%%%%%%%%%%%%%%%

\section{Finitely-additive measures on transversals and statement of the main theorem}

Recall that the ``Euclidean leaf,'' or the translation orbit, of a tiling $\Tk\in X_\om$ is the unstable set for the substitution map $\om$.
The stable leaf is a {\em transversal}, which we now define, and which is topologically 
a Cantor set for aperiodic tilings.
We then proceed to the construction of finitely-additive measures on the transversals. This construction is naturally dual to the one 
in the previous section.

\begin{defi} \label{def-trans}
For an admissible patch $P$ of tiles in the space $X_\om$ the set
$$
\Gam_{\om,P}:= \{\Tk\in X_\om:\ P \subset \Tk\}
$$
is called the transversal associated with the patch $P$.
\end{defi}

The tiling space $X_\om$ has a local product structure:
$$
X_\om \approx \Bigl(\bigcup_{j=1}^m (A_j \times \Gam_{\om,T_j})\Bigr)\mbox{\Large /}\sim\,,
$$
where $\approx$ is a natural homeomorphism, $T_j$ are the prototiles, $A_j=\supp(T_j)$, and the quotient $\sim$ corresponds to a certain ``gluing''
along the boundaries of tiles, see \cite{AP} for details. In fact, $X_\om$ can be considered as a 
translation surface or $\R^d$-solenoid \cite{BG,Gamb}.

If a patch $P\subset \Tk$ is such that $\supp(P)$ contains the origin in its interior, then
$\Gam_{\om,P}$ is a stable set of $\Tk$ for $\om$, in the sense that
$$
d(\om^k(\Tk'),\om^k(\Tk))\le c\lam^{-k}\ \ \mbox{for all}\ \Tk'\in \Gam_{\om,P},\ k\in \N,
$$
by the definition of the tiling metric $d$.

Now let us derive some properties of the transversals.
It is clear that 
\be \label{eq-trans1}
\om(\Gam_{\om,P}) \subset \Gam_{\om,\om(P)},\ \ \ \Gam_{\om,y+P} = y+\Gam_{\om,P}. 
\ee
We don't have equality in the inclusion above, however, we do have
\be \label{eq-trans2}
\Upsilon_k(\Gam_{\phi^k \om, \phi^k P}) = \Gam_{\phi^{k-1}\om, \phi^{k-1}P},
\ee
where $\Upsilon_k$ is the subdivision map from (\ref{eq-subdiv}).
Then how can we describe $\om(\Gam_{\om,P})$ precisely? This is the set of tilings $\Tk\in X_\om$ whose ``super-tiling'' $\Tk^{(1)}$ contains
the patch $\phi P$. Thus, we have
\be \label{decomp2}
\Gam_{\om,T_i} = \bigcup_{j=1}^m \bigcup_{x\in D_{ij}} (\om(\Gam_{\om,T_j}) - x),
\ee
where $D_{ij}$ are from (\ref{eq-dij}),
and this is a disjoint union.

Before we define the finitely-additive measures on the transversals, 
it is worthwhile to recall the formula for the unique invariant measure $\mu$. 
(We know that primitive self-affine tiling dynamical systems with finite local complexity are uniquely ergodic by
Theorem~\ref{unerg}.)
For a patch $P\subset \Tk\in X_\om$ and
$U\subset \R^d$, define
the set
$$
X_{P,U}:= \{\Sk\in X_\om:\, P-y\subset \Sk\ \mbox{for some}\ y\in U\}.
$$
Let $\eta>0$ be such that every prototile contains a ball of diameter $\eta$ in its interior.
It is clear (see e.g.\ \cite[Lemma 1.6]{So1}) that the sets $X_{P,U}$, with $\diam(U)\le \eta$ and $U$ open, 
generate the topology on the tiling space $X_\om$.
It is proved in \cite[Corollary 3.5]{So1} that the unique invariant measure $\mu$ satisfies
\be \label{eq-unerg}
\mu(X_{P,U})=\freq(P)\cdot\Lk^d(U)\ \ \mbox{for}\ P\subset \Tk\in X_\om\ \mbox{and $U$ Borel, with $\diam(U)\le \eta$},
\ee
where $\freq(P)$ is the uniform frequency of the patch $P$ in $\Tk$. (The existence of uniform frequencies is shown, e.g., in
\cite[Lemma A.6]{LMS2}.) In particular, we have
\be \label{eq-unerg2}
\mu(X_{T_j,U})=\freq(T_j)\cdot\Lk^d(U)
\ee
for a small enough $U$. It is well-known that 
\be \label{eq-PF}u^{(1)}:= (\freq(T_j))_{j\le m}
\ee 
is the Perron-Frobenius eigenvector of
the substitution matrix $\Sf$, normalized by the condition
$\langle v^{(1)},u^{(1)}\rangle=1$, where
$v^{(1)} = (\Lk^d(A_j))_{j=1}^m$ is a Perron-Frobenius eigenvector of $\Sf^t$.
Here and for the rest of the paper we are using the bilinear pairing in $\C^m$:
\be \label{pairing}
\langle v, u\rangle = \sum_{j=1}^m v_j u_j.
\ee

We should also note that there is a notion of {\em transverse measure} on a transversal $\Gam$.
It is a Borel measure $\nu$ on $\Bk(\Gam)$ such
that $\nu(A)=\nu(A-y)$ for every $A\in \Bk(\Gam)$ and $y\in \R^d$ such that $A-y\subset \Gam$. There is a 1-to-1 correspondence between
finite positive transverse measures and finite invariant measures for the tiling system, see \cite[Section 5]{BBG}.
In our case this is manifested by  (\ref{eq-unerg}) and (\ref{eq-unerg2}).

Next we proceed to define finitely-additive measures on the transversals $\Gam_{\om,T_i-x}$
for $i\le m$ and $x\in \R^d$. We can also define them on the transversals $\Gam_{\om,P}$ for more
general patches $P$, but that will not be necessary.

For $\Sf$ we have the direct-sum decomposition
$$
\C^m = \wtil{E}^+ \oplus \wtil{E}^-,
$$
where $\wtil{E}^+$ is spanned by Jordan cells of eigenvalues of $\Sf$ with absolute value greater than 1.
For $u \in \wtil{E}^+$, $j\le m$, $y\in \R^d$, and $k\ge 0$ let
\be \label{eq-defo}
\Phi^-_{u}(\om^k(\Gam_{\om,T_j-y})) = (\Sf^{-k} u)_j.
\ee
We have
$$
\om^k(\Gam_{\om,T_j-y}) \subset 
\Gam_{\om,T_i-x}\ \Longleftrightarrow\  T_i-x\in \om^k(T_j-y).
$$
We claim that for each $u \in \wtil{E}^+$,
(\ref{eq-defo}) defines a finitely-additive measure on the algebra of subsets of $\Gam_{T_i-x}$
generated by the sets $\om^k(\Gam_{\om,T_j-y})$, with
$j\le m$, $k\ge 0$ and $y\in \R^d$ such that $T_i-x\in \om^k(T_j-y)$.
It is enough to verify finite additivity in (\ref{decomp2}), since the general case reduces to it easily. We have
\begin{eqnarray*}
\Phi^-_{u}(\Gam_{\om,T_i}) = u_i = \sum_{j=1}^m \Sf_{ij} (\Sf^{-1}u)_j & = & \sum_{j=1}^m \Sf_{ij} \Phi^-_{u}(\om(\Gam_{\om,T_j}))\\
& = & \sum_{j=1}^m \sum_{x\in D_{ij}}\Phi^-_{u}(\om(\Gam_{\om,T_j}) - x),
\end{eqnarray*}
as desired.

We are not going to discuss the extension of $\Phi^-_{u}$ to a larger class of sets, but we will need finitely-additive measures
$m_{\Phi^-_{u}}$,
defined locally as the product $\Phi_u^-\times \Lk^d$, on the tiling space $X_\om$. 
In order to make this precise, we define the class of ``test functions'' which we will be dealing with, and we will
define their integrals with respect to $m_{\Phi^-_{u}}$.

\begin{defi} \label{good-func}
A function $f$ on
$X_\om$ is called {\em cylindrical} if it is integrable with respect to the unique invariant measure $\mu$ and depends only on the tile containing the origin,
that is,
$$
\exists\,i\le m,\ x\in \R^d,\ 0\in \supp(T_i)-x,\ T_i-x\in \Tk\cap\Tk'\ \Longrightarrow\ f(\Tk)=f(\Tk').
$$
\end{defi}

A cylindrical function may be identified with a family of functions
$\{\psi_i\}_{i\le m}$, where $\psi_i:\,A_i=\supp(T_i)\to \R$,
\ $\psi_i\in L^1(A_i)=L^1(A_i,\Lk^d)$ as follows:
$$
f(\Tk)=\psi_i(x)\ \ \mbox{if}\ 0\in A_i-x,\ \ T_i-x\in \Tk.
$$
The functions $\psi_i$ are only defined $\Lk^d$-a.e., which does not cause a problem since we will integrate 
cylindrical functions with 
respect to $\Phi^-_{u}\times \Lk^d$. 
The simplest cylindrical function is the characteristic function of a prototile $T_i$,
which is defined by $\psi_i\equiv 1$, $\psi_j\equiv 0$, $j\ne i$.

Now we define for any cylindrical $f$:
\be \label{int2}
m_{\Phi^-_{u}}(f):=\sum_{i=1}^m \Phi^-_{u}(\Gam_{\om,T_i}) \Lk^d(\psi_i) = \sum_{i=1}^m u_i \int_{A_i} \psi_i(y)\,dy.
\ee

\noindent{\bf Remarks.} 1. Finitely-additive measures $m_{\Phi^-_{u}}$, for $u\in \wtil{E}^+$, are invariant under the dynamics: this follows from
(\ref{int2}) and the fact that 
$$
\Phi^-_{u}(y+\Gam_{\om,P}) = \Phi^-_{u}(\Gam_{\om,y+P}) = \Phi^-_{u}(\Gam_{\om,P}).
$$

2. Let $u^{(1)}\in \C^m$ be the Perron-Frobenius eigenvector of the substitution matrix $\Sf$, 
normalized by the condition $\langle v^{(1)},u^{(1)}\rangle=1$.
As already mentioned, $u^{(1)}_j$ is the uniform frequency of tiles of type $j$ in the tilings $\Tk\in X_\om$. 
Thus, in view of (\ref{eq-unerg2}), $m_{\Phi^-_{u^{(1)}}}$ is
exactly the invariant probability measure $\mu$ on the tiling space. 
For a cylindrical function $f$ we denote by $\|f\|_1$ its norm in $L^1(X_\om,\mu)$; observe that
$$
\|f\|_1 = \sum_{i=1}^m u^{(1)}_i \int_{A_i}|\psi_i(y)|\,dy.
$$

3. Cylindrical functions are not dense in $L^1(X_\om,\mu)$; however, it follows from \cite[Lemma 1.6]{So1} that
the set of functions
$\{f\circ \om^{-k}:\ f\ \mbox{is cylindrical},\ k\in \N\}$ is dense.
Thus it is useful to compute $m_{\Phi^-_{u}}(f\circ \om^{-k})$ explicitly.
If $f$ is the characteristic function of a tile $T_j$, then $f\circ \om^{-k}$ is the characteristic function
of the super-tile $\phi^k T_j$, that is, $f\circ \om^{-k}(\Tk)=1$ if and only if $\Tk \in \om^k(\Gam_{\om,T_j})-y$ for
some $y\in \phi^k A_j$. Thus, it follows from
(\ref{eq-defo}) that
$$
m_{\Phi^-_{u}}(f\circ \om^{-k})=\sum_{i=1}^m (\Sf^{-k} u)_i\int_{\phi^k A_i} \psi_i\circ \phi^{-k}(y)\,dy
=\lam^{dk}\sum_{i=1}^m (\Sf^{-k} u)_i\int_{A_i} \psi_i (x)\,dx,
$$
keeping in mind that $|\det(\phi)|=\lam^d$.
In particular, if $\Sf u = \theta u$, then
\be \label{eq-defon}
m_{\Phi^-_{u}}(f\circ \om^{-k})=\theta^{-k} \lam^{dk} m_{\Phi^-_{u}}(f).
\ee

\medskip

Denote by $\wtil{E}^{++}$ the {\em rapidly expanding subspace} for the matrix $\Sf$, which is, by definitions, the linear span of Jordan cells for
$\Sf$ corresponding to eigenvalues
greater than $\theta_1^{\frac{d-1}{d}}=\lam^{d-1}$.
In our first main theorem, which we state below, only the finitely-additive measures $m_{\Phi^-_{u}}$ with $u$ in the rapidly expanding subspace 
play a role, since only their contribution dominates the ``boundary effects.'' %This is a phenomenon which only appears for $d\ge 2$.

Choose a basis $\{v^{(i)}\}_{i=1}^m$ for $\C^m$, consisting of eigenvectors and root vectors of $\Sf^t$, according to the ordering of the eigenvalues
$$
\theta_1=\lam^d> |\theta_2|\ge...\ge |\theta_m|
$$
(the eigenvalues are counted with algebraic multiplicity).
We set 
$$
v^{(1)} = (\Lk^d(A_j))_{j=1}^m,
$$
as discussed above. Then consider the dual basis $\{u^{(j)}\}_{j=1}^m$, so that $\langle v^{(i)}\,,\,u^{(j)}\rangle = \delta_{ij}$.
This agrees with the definition of $u^{(1)}$ in (\ref{eq-PF}). The vectors $\{u^{(j)}\}_{j=1}^m$ are the eigenvectors and root vectors of $\Sf$,
so that $\Sf u^{(j)}=\theta u^{(j)}$ if and only if $\Sf^t v^{(j)}=\theta v^{(j)}$ (note that we do not need to put complex conjugation, by our
definition of the pairing (\ref{pairing})).
Let $\ell$ be the dimension of the rapidly expanding subspace $E^{++}$, that is, 
$$
|\theta_\ell|>\theta_1^{\frac{d-1}{d}}\ \ \ \mbox{and}\ \ \ |\theta_{\ell+1}|\le \theta_1^{\frac{d-1}{d}}.
$$
Then $\{v^{(j)}\}_{j=1}^\ell$ is a basis for $E^{++}$ and $\{u^{(i)}\}_{i=1}^\ell$ is a basis for $\wtil{E}^{++}$.
Denote 
$$
\Phi^+_{j,\Tk}:= \Phi^+_{v^{(j)},\Tk}\ \ \mbox{and} \ \ \Phi^-_{i}:= \Phi^-_{u^{(i)}}.
$$

\begin{theorem} \label{th-main2} Let $(X_\om,\R^d)$ 
be a non-periodic self-similar tiling dynamical system of finite local complexity, let $\mu$ be the unique
invariant probability measure, and let $\Om$ be a bounded Lipschitz
domain in $\R^d$.
Then there exists a constant $C=C(\om,\Om)>0$, such that for any cylindrical function $f$ and any $\Tk\in X_\om$:
\begin{eqnarray} 
& & \left|\int_{\Om_R} f(\Tk - y)\,dy\ - \Lk^d(\Om_R) \int_{X_\om} f\,d\mu - 
\sum_{n=2}^\ell \Phi^+_{n,\Tk}(\Om_R)\cdot m_{\Phi^-_{n}}(f)\right|  \label{eq-th2} \\[1.2ex]
&  \le & \
CR^{d-1} (\log R)^s \|f\|_1, \ \ \mbox{for all}\ R\ge 2, \nonumber
\end{eqnarray}
where $s$ is the maximal size of the Jordan block corresponding to eigenvalues satisfying $|\theta|=\theta_1^{\frac{d-1}{d}}$ 
(if there are no such 
eigenvalues, then $s=0$). 
\end{theorem}

\noindent {\bf Remarks.} 
1. The second term in (\ref{eq-th2}) can be written in a way consistent with the sum that follows:
for a cylindrical $f$,
$$
\int_{X_\om} f\,d\mu=m_{\Phi^-_1}(f)\ \ \mbox{and}\ \ \Lk^d(\Om_R) = \Phi^+_{1,\Tk}(\Om_R).
$$

2. We can formally interpret (\ref{eq-th2}) also in the case when $|\theta_2|\le \theta_1^{\frac{d-1}{d}}$; then
$\ell=1$ and the sum in the formula (\ref{eq-th2}) is zero.

\medskip

It is not hard to extend (\ref{eq-th2}) to functions of the form $f\circ\om^{-k}$, where $k\in \N$ and $f$ is cylindrical, which
form a dense subset of $L^1(X_\om,\mu)$.
In the next corollary, for simplicity, we assume that 
$\Sf$ has no Jordan blocks in the rapidly expanding subspace and either $\Om$ is the ball centered at the origin, or $\phi$ is
a pure dilation.

\begin{corollary} \label{cor-main}
Under the assumptions of Theorem~\ref{th-main2}, suppose, in addition, that $\Sf$ has no Jordan blocks in $\wtil{E}^{++}$,
and the finitely-additive
measures $\Phi^+_{n,\Tk}, \Phi^-_{n}$ correspond to eigenvectors of $\Sf^t$ and $\Sf$ respectively, 
with eigenvalues $\theta_n$, for $n\le \ell$. Moreover, assume that either $\Om$ is the ball centered at the origin, or $\phi$ is
a pure dilation. 
Then we have for any cylindrical function $f$ and $k\in \N$:
\begin{eqnarray}
& & \left|\int_{\Om_R} f\circ \om^{-k}(\Tk - y)\,dy\ - \Lk^d(\Om_R) \int_{X_\om} f\circ \om^{-k}\,d\mu-\right.  \label{eq-corr} \\[1.2ex]
& &- \left.\sum_{n=2}^\ell \Phi^+_{n,\Tk}(\Om_R)\cdot m_{\Phi^-_{n}}(f\circ \om^{-k})\right| \nonumber
\le CR^{d-1} \lam^k (\log (\lam^{-k} R))^s \|f\circ \om^{-k}\|_1, 
\end{eqnarray}
for all $R\ge 2\lam^k$,
where $C$ is the constant from Theorem~\ref{th-main2}.
\end{corollary}

\begin{proof}
We have 
\begin{eqnarray}
\int_{\Om_R} f\circ \om^{-k}(\Tk - y)\,dy &  = & \int_{\Om_R} f(\om^{-k}\Tk - \phi^{-k}y)\,dy \nonumber \\
& = &
\lam^{dk} \int_{\phi^{-k} \Om_R} f(\om^{-k}\Tk - x)\,dx. \label{extra}
%& = & \lam^{dk} \int_{\Om_{\lam^{-k}R}} f(\om^{-k}\Tk - x)\,dx, 
\end{eqnarray} 
Observe that
$$
\phi^{-k} \Om_R=\Om_{\lam^{-k}R}
$$
by the assumption on $\Om$ and $\phi$, so we can apply (\ref{eq-th2}), with $\Tk$ replaced by $\om^{-k}\Tk$ and
$R$ replaced by $\lam^{-k}R$. We have
$\Lk^d(\Om_{\lam^{-k}R})=\lam^{-dk}\Lk^d(\Om_R)$, 
$$
\Phi^+_{n,\om^{-k}\Tk}(\phi^{-k} \Om_R)=\theta_n^{-k}\Phi^+_{n,\Tk}(\Om_R)
$$
by (\ref{coc3}), and
$$
m_{\Phi^-_{n}}(f) = \theta_n^k \lam^{-kd} m_{\Phi^-_{n}}(f\circ \om^{-k})
$$
by (\ref{eq-defon}). Since everything is multiplied by $\lam^{dk}$ from (\ref{extra}), all ``extra'' factors cancel out.
In the right-hand side of (\ref{eq-th2}) we will get $C(\lam^{-k}R)^{d-1} (\log (\lam^{-k} R))^s \|f\|_1$,
which is also multiplied by $\lam^{dk}$, and keeping in mind that $\mu$ is $\om$-invariant by Lemma~\ref{lem-invmeas},
we obtain (\ref{eq-corr}).
\end{proof}

Next we deduce upper deviation bounds from  Theorem~\ref{th-main2}.

\begin{corollary} \label{th-main1}
Let $(X_\om,\R^d,\mu)$ be a non-periodic
self-similar tiling dynamical system. Suppose that the substitution matrix ${\sf S}$
has eigenvalues $\th_1,\ldots,\th_m$ (real and complex), counted with multiplicities and ordered in such a way that
$\th_1 > |\th_2|\ge \cdots \ge |\th_m|$.
Further, let $s$ be the size of the largest Jordan block associated with the
eigenvalues of absolute value $|\th_2|$.

Given a bounded Lipschitz domain $\Om$, there exists a constant $\wtil{C}=\wtil{C}(\om,\Om)>0$ 
such that for any cylindrical function $f$, with $\|f\|_1=1$, any
tiling $\Tk\in X_\om$, and $R\ge 2$ we have
$$
\left|\int_{\Om_R} f(\Tk-y)\,dy - \Lk^d(\Om_R) \int_{X_\om} f\,d\mu\right|
              \le \left\{ \begin{array}{ll} \wtil{C} R^{d-1}, & \mbox{if}\ |\th_2| < \th_1^{\frac{d-1}{d}}; \\
                                            \wtil{C}R^{d-1} (\log R)^s, & \mbox{if}\ |\th_2| = \th_1^{\frac{d-1}{d}}; \\
                                            \wtil{C}R^\alpha (\log R)^{s-1}, & \mbox{if}\ |\th_2| >  \th_1^{\frac{d-1}{d}},
                \end{array} \right.
$$
where $\alpha = d\log|\th_2|/\log\th_1 \in (d-1,d)$.
\end{corollary}

\begin{proof} The first two cases are immediate from (\ref{eq-th2}), since then $\ell=1$. The third case also follows from (\ref{eq-th2}), in view of
Lemma~\ref{lem-finadd2}. \end{proof}

It is possible to show that Corollary~\ref{th-main1} is sharp, at least in the special case when the tiles are polyhedral,
in the sense that
the powers of $R$ in the right-hand side in each case cannot be replaced by a smaller power.

\medskip

\noindent {\bf Remarks.} 1. There are a number of results related to Corollary~\ref{th-main1} in the literature.
When $f$ is assumed to be the characteristic function of a  prototile, the corollary reduces to estimates of the rate of
convergence to frequency for prototiles. In the case $d=1$ this is essentially the same as estimating symbolic discrepancy for substitutions, which was
done by Adamczewski \cite{Adamcz}.

Solomon \cite{Solomon1,Solomon2} gives deviation estimates similar to ours for the number of tiles in a ``super-tile'' of high order.
The tiles are assumed to be
bi-Lipschitz equivalent to a ball, but the substitution need not be non-periodic. Under these assumptions, the estimates are shown to be sharp.

Aliste-Prieto, Coronel and Gambaudo \cite{ACG,ACG2} obtain analogous deviation estimates. The paper \cite{ACG}, which deals with the
$d=2$ case, estimates the deviation of average from the frequency for general Jordan domains and for
very general substitution tilings, including non-FLC tilings, the ``pinwheel-like'' tilings and tiles with fractal
boundary. However, the extension to $d>2$ in \cite{ACG2} handles only
the case of ``small'' $\theta_2$ under the stronger assumption $|\theta_2|\le \theta_1^{\frac{1}{d}}$.

Interest in such estimates was inspired by questions on bi-Lipschitz equivalence and bounded displacement of separated nets (also called Delone sets)
arising from primitive substitutions, like the Penrose tiling, to the lattice in $\R^d$, see \cite{BuKl}.

2. Sadun \cite{Sadun} obtained deviation estimates for the number of patches per volume in balls of large radius 
using rational \v{C}ech cohomology, with an error term computable from the patterns that appear on the boundary.

%%%%%%%%%%%%%%%%%%%%%%%%%%%%%%%%%%%%%%%%%%%%%%%%%%%%%%%%%%%%%%%%%%%%%%%%%%%%%%%%%%%%%%%%%%%%%%%%%%%%%%%%%%%

\section{Proof of Theorem~\ref{th-main2}.}

We will use the following notation: for a set $E\subset \R^d$, a tiling $\Tk$, and a patch $P$, denote by $N_P(E,\Tk)$ the number of
translated copies of $P$ in the tiling $\Tk$ whose support is contained in $E$. Since $\Tk$ is now fixed, we will just write $N_P(E)=N_P(E,\Tk)$.

Writing the cylindrical $f$ as a sum over prototiles $i\le m$, we can assume without loss of generality
that $\psi_j\equiv 0$ for $j\ne i$, and let $\psi:=\psi_i$.
Denote
$$
\Ik:=\int_{\Om_R} f(\Tk - y)\,dy.
$$
It follows from the definition of $f$ that if $y$ belongs to a translate of $T_i$ in $\Tk$, that is, $y\in A_i-x$ and $T_i-x\in \Tk$
for some $x\in \R^d$, then
$$f(\Tk-y)=\psi(y+x),$$ and $f(\Tk-y)=0$ otherwise. Thus,
\be \label{eq-integral}
\Ik = \sum_{x: \ (A_i-x)\cap \Om_R\ne \es} \int_{(A_i-x)\cap \Om_R} \psi(y+x)\,dy,
\ee
where the sum is over $x$ such that $T_i-x\in \Tk$.
Every translate of $T_i$ which is contained in $\Om_R$ contributes $\Lk^d(\psi)=\int_{A_i} \psi(y)\,dy$ to $\Ik$, 
and every translate of $T_i$ which intersects
the boundary of $\Om_R$ contributes at most $\|\psi\|_1=(u_i^{(1)})^{-1}\|f\|_1$.
Notice that the number of the translates intersecting the boundary does not exceed
$
\Lk^d(U(\partial \Om_R, d_{\max})) a_{\min}^{-1}.
$
We can write
$$
\Lk^d(U(\partial \Om_R, d_{\max})) =R^d\Lk^d(U(\partial \Om, d_{\max}/R)) \le C(\partial \Om, 1) d_{\max} R^{d-1},\ \ 
\mbox{for}\ R> d_{\max},
$$
by (\ref{eq-gmt}), hence
\be \label{116}
\Ik = N_{T_i}(\Om_R) \Lk^d(\psi) + O(R^{d-1}\|f\|_1),
\ee
where the implied constant in $O(\cdot)$ depends only on $\Om$
and $\om$. Thus it suffices to prove the desired estimate for $N_{T_i}(\Om_R)$.
(Note that $|\Lk^d(\psi)|\le (u_i^{(1)})^{-1}\|f\|_1$, so we will get the factor of $\|f\|_1$ in the right-hand
side of (\ref{eq-th2}).)

By the definition of the substitution matrix $\Sf$, we have
\be \label{106}
\om^k(T_j)-y\in \Tk\ \Rightarrow\ N_{T_i}(\phi^k(A_j)-y) = \Sf^k(i,j)=(\Sf^t)^k(j,i)=\langle (\Sf^t)^k e^{(i)}\,,\,e^{(j)}\rangle,
\ee
where $e^{(i)}$ is the standard $i$-th basis vector.

Recall that we have chosen a basis $\{v^{(n)}\}_{n=1}^m$ for $\C^m$, such that $\{v^{(n)}\}_{n=1}^\ell$ is a basis for
the $\Sf^t$-invariant subspace $E^{++}$, and a dual basis $\{u^{(n)}\}_{n=1}^m$.
Then we have
$$
e^{(i)} = \sum_{n=1}^m \langle e^{(i)}\,,\, u^{(n)}\rangle \,v^{(n)}= \sum_{n=1}^m u^{(n)}_i v^{(n)}.
$$
Therefore,
\be\label{102}
\langle (\Sf^t)^k e^{(i)}\,,\,e^{(j)}\rangle = \sum_{n=1}^m u^{(n)}_i ((\Sf^t)^k  v^{(n)})_j.
\ee
Next we essentially repeat the construction of Lemma~\ref{lem-finadd1} and consider the set
$ \Rk^{(k)}=\Rk^{(k)}(\Om_R) $ defined by (\ref{def-rk}).
Further, let us write $\Rk^{(k)} = \bigcup_{j=1}^m \Rk^{(k)}_j$, where $\Rk^{(k)}_j$ is the set of tiles of order $k$ in $\Rk^{(k)}$ of type $j$.
 Let $k_R=\max\{k:\ \Rk^{(k)}\ne \es\}$.
We have, in view of (\ref{106}) and (\ref{102}),
\begin{eqnarray}
N_{T_i}(\Om_R) & = & \sum_{k=0}^{k_R} N_{T_i} (\supp(\Rk^{(k)})) \nonumber \\
               & = & \sum_{k=0}^{k_R} \sum_{j=1}^m \#\Rk^{(k)}_j 
                      \sum_{n=1}^m u^{(n)}_i ((\Sf^t)^k  v^{(n)})_j \nonumber \\
     & = & \left(\sum_{n=1}^\ell + \sum_{n=\ell+1}^m \right) u^{(n)}_i \sum_{k=0}^{k_R} \sum_{j=1}^m 
                      \#\Rk^{(k)}_j ((\Sf^t)^k  v^{(n)})_j \label{1071} \\
     & =:& I_1 + I_2. \nonumber 
\end{eqnarray}

Recall that
$$
\Phi^+_{n,\Tk}(\supp(T)) = ((\Sf^t)^k  v^{(n)})_j\ \ \mbox{for}\ T\in \Tk^{(k)}\ \mbox{of type}\ j.
$$
Using this and finite-additivity of $\Phi^+_{n,\Tk}$, we can write
$$
I_1 = \sum_{n=1}^\ell  u^{(n)}_i \Phi^+_{n,\Tk}(\supp(\Tk|_{\Om_R})).
$$
By Lemma~\ref{lem-finadd2.cor},
\be \label{shows}
\Phi^+_{n,\Tk}(\supp(\Tk|_{\Om_R}))=\Phi^+_{n,\Tk}(\Om_R)+O(R^{d-1})\ \ \mbox{for}\ n\le \ell,
\ee
where the implied constant depends only on $\Om$ and $\om$.
Recall that $u^{(n)}_i= \Phi_n^-(\Gam_{\om,T_i})$.
Thus (\ref{shows}) yields
\be \label{117}
N_{T_i}(\Om_R)\Lk^d(\psi) = \sum_{n=1}^\ell \Phi^+_{n,\Tk}(\Om_R) \cdot m_{\Phi_n^{-}}(f) + I_2\cdot \Lk^d(\psi)+O(R^{d-1}),
\ee
with the implied constant that depends only on $\Om$ and $\om$.

It remains to estimate $I_2$. We have
$$
|I_2| \le \sum_{n=\ell+1}^m \|u^{(n)}\| \sum_{k=0}^{k_R} \# \Rk^{(k)} \|(\Sf^t)^k v^{(n)}\|.
$$
Below we use the notation $\lesssim$ to indicate inequality up to a multiplicative constant that depends only on $\Om$ and $\om$.
We have 
$$
\# \Rk^{(k)}= \# \Rk^{(k)}(\Om_R) \lesssim R^{d-1} \lam^{-(d-1)k}
$$
by (\ref{103}), and
\be \label{111}
\|(\Sf^t)^k v^{(n)}\| \lesssim k^{s-1}\lam^{(d-1)k}\ \ \mbox{for}\ n\ge \ell+1,\ k>0,
\ee
by the assumption that $v^{(n)}$, with $n\ge \ell+1$, is in the invariant subspace of $\Sf^t$ corresponding to eigenvalues $\theta$,\ 
$|\theta|\le \lam^{d-1}$, and $s$ is the maximal size of the Jordan block of an eigenvalue $\theta$,\ $|\theta|= \lam^{d-1}$.
It follows  that
$$
|I_2| \lesssim R^{d-1}\sum_{k=0}^{k_R} k^{s-1} \lesssim R^{d-1}k_R^s \lesssim R^{d-1}(\log R)^s,
$$
where the last inequality follows from (\ref{eq-kr}). This,
together with (\ref{116}) and (\ref{117}),
completes the proof of (\ref{eq-th2}) in the case when $s\ge 1$. If $s=0$, that is, all remaining eigenvalues are less than 
$\lam^{d-1}$
in absolute value, then we can replace the right-hand side of (\ref{111}) by $\gam^k$ for some $\gam< \lam^{d-1}$, and use that
$\sum_{k=0}^{k_R} \lam^{-k(d-1)} \gam^k \asymp 1$ to obtain
\be \label{eq-rapid}
|I_2|\lesssim R^{d-1}.
\ee
Now the theorem is proved completely. \qed

%%%%%%%%%%%%%%%%%%%%%%%%%%%%%%%%%%%%%%%%%%%%%%%%%%%%%%%%%%%%%%%%%%%%%%%%%%%%%%%%%%%%%%%%%%%%%%%%%%%%

\section{Limit laws for the deviation of ergodic averages} \label{section_limitlaw}

In order to obtain the limit law, we need to make the following additional assumptions:

\medskip

{\bf (A)} {\em the expansion map of the tiling substitution is a pure dilation: $\phi(x) = \lam x,\ \lam>1$;}

\medskip

{\bf (B)} {\em all the $\Tk$-prototiles are polyhedral.}

\medskip

Denote by $\Fk$ the class of {\bf bounded} cylindrical functions on $X_\om$.
For any $f\in \Fk$ and $\Tk\in X_\om$, 
define a continuous function on $[0,1]$ by
\be \label{eq-rv1}
\Frs_n[f,\Tk](r) = \int_{Q_{r\lam^n}} f(\Tk-y)\,dy.
\ee
Recall that $Q_r = [-r/2,r/2]^d$.
We consider $r\mapsto \Frs_n[f,\Tk](r)$ as a random variable on $(X_\om,\mu)$ with the values in $C[0,1]$, endowed with the norm topology.

\begin{theorem} \label{th-limitlaw}
Let $(X_\om,\R^d,\mu)$ be a non-periodic self-similar tiling dynamical system 
satisfying the assumptions {\bf (A)} and {\bf (B)}.
Suppose that the substitution matrix $\Sf$ has a simple positive real second eigenvalue $\theta_2 > 
\lam^{d-1}=\theta_1^{\frac{d-1}{d}}$, and all
other eigenvalues are less than $\theta_2$ in absolute value. Then there is a continuous functional $\beta:\,\Fk \to \R$ and a compactly supported
non-degenerate measure $\nu$ on $C[0,1]$ such that for any $f\in \Fk$ satisfying $\int_{X_\om} f\,d\mu=0$ and $\beta(f)\ne 0$, the
sequence of random variables
$$
\frac{\Frs_n[f,\Tk]}{\beta(f) \theta_2^n}
$$
converges in distribution to $\nu$ as $n\to\infty$.
\end{theorem}

\noindent {\bf Remarks.} Nondegeneracy of the measure means that if 
$\varphi\in C[0,1]$ 
is distributed according to $\nu$, then for any $r_0\in (0,1]$ the distribution of the real-valued random variable $\varphi(r_0)$
is not concentrated at a single point. 

The measure $\nu$ and the functional $\beta(f)$ naturally come from the 2-nd term in the formula (\ref{eq-th2}), with $\Om=Q_1$, 
since the 1-st term in (\ref{eq-th2}) is zero. In other words,
\be \label{def-eta}
\nu= \mbox{the distribution of $r\mapsto \Phi^+_{2,\Tk}(Q_r)$,\ $r\in [0,1]$},
\ee
as a random variable on $(X_\om,\mu)$, and 
\be \label{eq-betaf}
\beta(f)= m_{\Phi^-_2}(f).
\ee
Note that 
$$
|\beta(f)|\le \sum_{i=1}^m |u^{(2)}_i|\cdot \Lk^d(A_i)\cdot \|f\|_\infty
$$
by (\ref{int2}),
so $\beta$ is a continuous functional on $\Fk\subset L^\infty(X_\om)$.

The theorem in the case $d=1$ was established in \cite{Bufetov}, and the general scheme of our proof is similar. However, it should be emphasized
that there are many complications because of the ``boundary effects'' for $d\ge 2$. Note that the assumptions {\bf (A)} and {\bf (B)}
hold in the one-dimensional case (with connected tiles) automatically.

\medskip

\noindent {\em Proof.} 
We are going to use a basic result (see \cite[Th.7.1]{Billing} or \cite{Bogachev}) which says that, given a sequence of probability measures
on $C[0,1]$, if their finite-dimensional distributions converge and the sequence is tight, then the measures converge weakly, which is
equivalent to saying that the random variables converge in distribution. Recall that
a family of probability measures on a separable metric space
is {\em tight} if for every $\eps>0$ there is a compact set such that its complement has measure less than $\eps$
for every measure in the family.

In view of (\ref{eq-th2}) and (\ref{eq-hoelder}), we have for $f\in \Fk$, with $\int_{X_\om} f\,d\mu=0$, by the assumptions on the substitution matrix:
$$
\left|\int_{Q_R} f(\Tk-y)\,dy - \Phi^+_{2,\Tk}(Q_R)\cdot m_{\Phi^-_2}(f)\right| \le 
C(\om,\Om) R^{\alpha-\delta}\|f\|_1,\ \ \mbox{with}\ \alpha=\frac{\log\theta_2}{\log\lam},
$$
for some $\delta\in (0,\alpha)$ and all $R\ge 2$. Therefore, for $f\in \Fk$ with $\beta(f)\ne 0$, by (\ref{eq-betaf}),
\be \label{eq-distrib1}
\left|\frac{\Frs_n[f,\Tk](r)}{\beta(f) \theta_2^n} - \frac{\Phi^+_{2,\Tk}(Q_{r\lam^n})}{\theta_2^n}\right|
\le C(\om,\Om) \lam^{-\delta n},\ \ \mbox{for all $n\in \N$ and $r\in [0,1]$.}
\ee
Note the following important equality, which follows from (\ref{coc3}) and the fact that $\phi(Q_r) = \lam Q_r = Q_{r\lam}$
by the assumption {\bf (A)}:
$$
\Phi^+_{2,\om(\Tk)}(Q_{r\lam}) = \theta_2 \Phi^+_{2,\Tk}(Q_{r}).
$$
Thus,
$$
\frac{\Phi^+_{2,\Tk}(Q_{r\lam^n})}{\theta_2^n} = \Phi^+_{2,\om^{-n}(\Tk)}(Q_r).
$$
Observe that $r\mapsto \Phi^+_{2,\om^{-n}(\Tk)}(Q_r)$ has the distribution of $\nu$ from (\ref{def-eta}) for all $n$, since
$\mu$ is $\om^{-1}$-invariant by Lemma~\ref{lem-invmeas}.
Thus, it follows from (\ref{eq-distrib1}) that the $k$-dimensional distributions of 
$\frac{1}{\beta(f)\theta_2^n}(\Frs_n[f,\Tk](r_1), \ldots, \Frs_n[f,\Tk](r_k))$ converge weakly to the  $k$-dimensional distributions of
$(\Phi^+_{2,\Tk}(Q_{r_1}), \ldots, \Phi^+_{2,\Tk}(Q_{r_k}))$. Further, (\ref{eq-hoelder2}) in Lemma~\ref{lem-finadd3} shows that the
support of $\nu$ is compact in $C[0,1]$ by the Arzel\`a-Ascoli Theorem.
In order to complete the proof, we need to establish (i) tightness; (ii) nondegeneracy of the limit measure $\nu$.

\subsection{Tightness.}
The following lemma will imply that the sequence of distributions of $r\mapsto \theta_2^{-n} \Frs_n[f,\Tk](r)$ is tight, again by  Arzel\`a-Ascoli.
In fact, all the distributions are supported on a single compact set.

\begin{lemma} \label{lem-tight}
There exists $C(\om)$ and $n_0\in \N$ such that for all  $f\in \Fk$ with $\int f\,d\mu=0$, for
all $\Tk\in X_\om$, all $n\ge n_0$, and all $r_1,r_2\in [0,1]$,
\be \label{eq-tight}
\frac{|\Frs_n[f,\Tk](r_2) - \Frs_n[f,\Tk](r_1)|}{\theta_2^n} \le C(X_\om)\|f\|_\infty \cdot |r_2-r_1|^{\alpha-(d-1)},
\ee
where $\alpha = \frac{\log \theta_2}{\log \lam}$.
\end{lemma}

\begin{proof}
Let $r_1 < r_2$. 
We have
$$
\Frs_n[f,\Tk](r_2) - \Frs_n[f,\Tk](r_1) = \int_{Q_{\lam^n r_2} \setminus Q_{\lam^n r_1}} f(\Tk-y) \,dy =:\Ik.
$$
By the definition of cylindrical functions, there exist $\psi_i \in L^\infty(A_i)$, $i\le m$, such that
$$
f(\Tk)= \psi_i(x)\ \ \mbox{iff}\ \ 0\in \Int(A_i-x),\ T_i-x\in \Tk
$$
(we can, of course, ignore the case when $x$ belongs to the boundary of a tile, since the boundary has measure zero).
Then we have, similarly to (\ref{eq-integral}):
$$
\Ik =\sum_{i=1}^m\ 
 \sum_{x:\, (A_i-x) \cap (Q_{\lam^n r_2} \setminus Q_{\lam^n r_1})\ne \es}\ \ \int_{(A_i-x) \cap (Q_{\lam^n r_2} \setminus Q_{\lam^n r_1})}
\psi_i(y+x)\,dy,
$$
where the inside sum is over $x$ such that $T_i-x\in \Tk$.
For $n\in \N$ such that $\lam^n(r_2-r_1) \le 1$, we estimate $\Ik$ as follows, keeping in mind that $\|f\|_\infty = \max_i \|\psi_i\|_\infty$:
\begin{eqnarray}
|\Ik|\le \Lk^d(Q_{\lam^n r_1} \setminus Q_{\lam^n r_2}) \cdot\|f\|_\infty & = & [(\lam^n r_2)^d-(\lam^n r_1)^d]\cdot\|f\|_\infty \nonumber\\
              & = & \lam^{nd}(r_2^d-r_1^d) \cdot\|f\|_\infty \nonumber \\
              & \le &  d\lam^{nd}(r_2-r_1) \cdot\|f\|_\infty. \label{eq-dur1}
\end{eqnarray}
Observe that
$$
\left(\frac{\lam^d}{\theta_2}\right)^n = \lam^{n(d - \frac{\log\theta_2}{\log\lam})}\le (r_2-r_1)^{\alpha-d},
$$
by the assumption $\lam^n \le (r_2-r_1)^{-1}$, keeping in mind that $\alpha = \frac{\log\theta_2}{\log\lam}$. Thus, by (\ref{eq-dur1}),
$$
\frac{|\Ik|}{\theta_2^n} \le d\|f\|_\infty \left(\frac{\lam^d}{\theta_2}\right)^n (r_2-r_1) \le d\|f\|_\infty (r_2-r_1)^{\alpha-(d-1)},
$$
which yields (\ref{eq-tight}) for such $n$. 

For $n\in \N$ such that $\lam^n(r_2-r_1) >1$, we proceed similarly to the proof of Theorem~\ref{th-main2} and estimate
\begin{eqnarray*}
& & \left|\Ik - \sum_{i=1}^m N_{T_i}(Q_{\lam^n r_2} \setminus Q_{\lam^n r_1})\cdot \Lk^d(\psi_i)\right|\\
& \le & \Lk^d(U(\partial Q_{\lam^n r_1}\cup \partial Q_{\lam^n r_2}, d_{\max}))
\cdot a_{\min}^{-1} a_{\max} \|f\|_\infty,
\end{eqnarray*}
where $a_{\max}$ is the maximal volume of a $\Tk$ prototile.
By (\ref{eq-claim}), 
\begin{eqnarray*}
\Lk^d(U(\partial Q_{\lam^n r_1}\cup \partial Q_{\lam^n r_2}, d_{\max})) & \le & d\, 2^{d+1} d_{\max} ((\lam^n r_1)^{d-1} + (\lam^n r_2)^{d-1})\\
                & \le & d\, 2^{d+2} d_{\max} \lam^{n(d-1)}.
\end{eqnarray*}
We have
$$
\left(\frac{\lam^{d-1}}{\theta_2}\right)^n = \lam^{-n(\alpha-(d-1))}< (r_2-r_1)^{\alpha-(d-1)}
$$
by the assumption $\lam^n(r_2-r_1) >1$. Therefore, 
$$
\theta_2^{-n}
\left|\Ik - \sum_{i=1}^m N_{T_i}(Q_{\lam^n r_2} \setminus Q_{\lam^n r_1})\cdot \Lk^d(\psi_i)\right| 
\le \const\cdot \|f\|_\infty\cdot (r_2-r_1)^{\alpha-(d-1)},
$$
with the constant depending only on $X_\om$,
and it remains to estimate 
$$
\theta_2^{-n} \sum_{i=1}^m N_{T_i}(Q_{\lam^n r_2} \setminus Q_{\lam^n r_1}) \cdot \Lk^d(\psi_i).
$$
This is done similarly to (parts of) the proof of Theorem~\ref{th-main2}, with some elements from the proof of Lemma~\ref{lem-finadd3}.
We proceed to the formal estimate.

Consider $\Rk^{(k)} = \Rk^{(k)}(Q_{\lam^n r_2} \setminus Q_{\lam^n r_1})$, and let $k_0= \max\{k:\  \Rk^{(k)}\ne \es\}$. Further, let
$\Rk_j^{(k)}$ be the collection of tiles of type $j$ in $\Rk^{(k)}$.
For $i\le m$, using (\ref{106}) and (\ref{102}), we have,
similarly to (\ref{1071}),
\begin{eqnarray*}
N_{T_i}(Q_{\lam^n r_2} \setminus Q_{\lam^n r_1}) & = & \sum_{k=0}^{k_0} N_{T_i} (\supp(\Rk^{(k)}))  \\
               & = & \sum_{k=0}^{k_0} \sum_{j=1}^m \#\Rk^{(k)}_j
                      \sum_{s=1}^m u^{(s)}_i ((\Sf^t)^k  v^{(s)})_j \\
     & = & \left(\sum_{s=1}^1 + \sum_{s=2}^m \right) u^{(s)}_i \sum_{k=0}^{k_0} \sum_{j=1}^m
                      \#\Rk^{(k)}_j ((\Sf^t)^k  v^{(s)})_j \\
     & =:& I_1^{(i)} + I_2^{(i)}. 
\end{eqnarray*}
Now, 
$$
I_1^{(i)} = u_i^{(1)}\Lk^d(\supp(\Tk|_{Q_{\lam^n r_2} \setminus Q_{\lam^n r_1}})),
$$
hence
$$
\sum_{i=1}^m I_1^{(i)}\cdot \Lk^d(\psi_i) =0,
$$
in view of $\int f\,d\mu = \sum_{i=1}^m u_i^{(1)} \cdot \Lk^d(\psi_i)=0$.
Next,
\begin{eqnarray*}
|I_2^{(i)}| & \le & \sum_{s=2}^m \|u^{(s)}\|\sum_{k=0}^{k_0} \#\Rk^{(k)}\|(\Sf^t)^k v^{(s)}\| \\
            & \le & C_4 \sum_{k=0}^{k_0} \#\Rk^{(k)} \theta_2^k,
\end{eqnarray*}
by the assumptions on the matrix $\Sf^t$, where the constant $C_4>0$ depends only on the tiling space.
We have
$$
\#\Rk^{(k)} \le C(d,X_\om)(\lam^n r_2)^{d-1} \lam^{-(d-1)k} \le C(d,X_\om) \lam^{(d-1)(n-k)}
$$
by (\ref{108}), hence
\be \label{last1}
\theta_2^{-n} \left| \sum_{i=1}^m I_2^{(i)}\cdot \Lk^d(\psi_i)\right| \le C_5\left(\frac{\theta_2}{\lam^{d-1}}\right)^{k_0-n}\|f\|_\infty,
\ee
with a constant $C_5>0$ that depends only on the tiling space.
Recall that $\eta\lam^{k_0} \le \lam^n(r_2-r_1)$, where $\eta$ is the radius of a ball contained in every $\Tk$ prototile.
Thus $\lam^{k_0-n} \le \frac{r_2-r_1}{\eta}$, hence
the right-hand side of (\ref{last1}) is bounded above by
$$
C_5 \left(\frac{r_2-r_1}{\eta} \right)^{\alpha-(d-1)}.
$$
Now, combining everything together, we obtain the desired estimate.
\end{proof}

\subsection{Nondegeneracy of the limiting measure.}
It remains to prove that $\nu$ is non-trivial and non-degenerate for every $r\in (0,1]$. Assume, to the contrary, that for some $r$
we have $\Phi^+_{2,\Tk}(Q_r)=c$ for $\mu$-a.e.\ $\Tk\in X_\om$. By Fubini, we can find $\Tk\in X_\om$ such that
\be \label{kuk1}
\forall\,x\in \Q^d,\ \forall\,n\in \Z,\ \Phi^+_{2,\om^{-n}(\Tk-rx)}(Q_r)= c.
\ee
Here we use that $\mu$ is invariant under translations and under the action of $\om^{-1}$.
By (\ref{coc1}), we obtain that $\Phi^+_{2,\Tk}(Q_r+x)=c$ for all $x\in \Z^d$, and then by finite additivity,
$$
\Phi^+_{2,\Tk}(Q_{kr})=k^2 \Phi^+_{2,\Tk}(Q_r)=k^2c\ \ \mbox{for}\ k\in \N,
$$
decomposing the larger cube into the union of disjoint translates of $Q_r$. On the other hand,
$$
\Phi^+_{2,\Tk}(Q_{\lam^n r}) = \theta_2^n \Phi^+_{2,\om^{-n}(\Tk)}(Q_r)=\theta_2^n c
$$
by (\ref{coc3}). Now take $k=\lfloor \lam^n \rfloor$ and observe that
$$
|\Phi^+_{2,\Tk}(Q_{\lam^n r})-\Phi^+_{2,\Tk}(Q_{kr})|\le \const\cdot \lam^{n(d-1)}
$$
by (\ref{eq-hoelder2}). 
This implies that $c=0$; otherwise, we get a contradiction for $n$ sufficiently large, keeping in mind that
$\lam^{d-1}< \theta_2$.

Now suppose  $c=0$. Then $\Phi^+_{2,\Tk}(Q_{k^{-1}r}-rx)=0$ for $k\in \Nat$ and $x\in \Q^d$ by the argument as above. Then we can approximate
supports of the tiles of $\Tk$ by the unions of such cubes to conclude that they also have zero $\Phi^+_{2,\Tk}$-measure. But this is a 
contradiction, since $\Phi^+_{2,\Tk}(A_i-y) = v_i$, the $i$-th component of the eigenvector of $\Sf^t$ corresponding to $\theta_2$, if
$T_i-y\in \Tk$.

Let us explain this more carefully. It is only here that we are using the assumption that the prototiles are polyhedral. Fix a tile $T_i-y\in \Tk$ and
denote by $\Om_n$ the union of ``grid cubes'' $2^{-n}(Q_r-rx)$, with $x\in \Z^d$, whose closure is contained in the interior of 
$A_i-y$.
Then $V_n:=(A_i-y)\setminus \Om_n$ is a Lipschitz domain and $\Phi^+_{2,\Tk}(\Om_n)=0$ by the argument above.
We essentially repeat the arguments from Lemma~\ref{lem-finadd2} and Lemma~\ref{lem-finadd3} and start by writing
\be \label{eq-ult2}
\Phi^+_{2,\Tk}(V_n) = \sum_{k=-\infty}^{k_0} \sum_{T\in \Rk^{(k)}(V_n)}\Phi^+_{2,\Tk}(\supp(T)),
\ee
where $k_0 = \max\{k:\ \Rk^{(k)}(V_n)\ne \es\}$. Next,
\be \label{eq-ult1}
\#\Rk^{(k)}(V_n) \le \Lk^d(U(\partial V_n, d_{\max}\lam^{k+1}))a_{\min}^{-1} \lam^{-dk}.
\ee
By construction, $\Int(A_i-y) \subset U(\Om_n,2^{-n}r\sqrt{d})$, hence
\be \label{eq-ult3}
\lam^{k_0}\eta \le 2^{-n}r\sqrt{d},
\ee
where $\eta$ is the diameter of a ball contained in every $\Tk$ prototile, thus $d_{\max}\lam^{k+1}\le b_1\cdot 2^{-n}r$
for $k\le k_0$ for some $b_1$ independent of $n$. An elementary argument (see \cite[Lemma 2.2]{Lacz})
 shows that for any union $F$ of lattice cubes in $Z^d$ we have
\be \label{eq-gmt2}
\Lk^d(U(\partial F,t))\le 2(1+2b_1)^{d-1}t\,\Hk^{d-1}(\partial F),\ \ \ t\in (0,b_1],
\ee
where $\Hk^{d-1}(\partial F)$ is just the surface area of the boundary. Indeed, for every face of $\partial F$ (say, with the
``vertical'' normal), consider
the ``parallelepiped neighborhood'' of the face, with the vertical side length equal to $2t$ and the other $(d-1)$
sides of length $1+2b_1$. 
Clearly, it contains the Euclidean neighborhood of the face of radius $t$ for all $t\le b_1$, and the inequality (\ref{eq-gmt2})
follows. Scaling by $2^{-n}r$, we obtain
$$
\Lk^d(U(\partial \Om_n,t))\le 2(1+2b_1)^{d-1} t \,\Hk^{d-1}(\partial \Om_n),\ \ \ t\in (0, b_1\cdot 2^{-n}r].
$$
Therefore, for large $n$, such that $d_{\max}\lam^{k_0+1}\le 1$, we have, in view of (\ref{eq-gmt}),
\begin{eqnarray*}
\Lk^d(U(\partial V_n,d_{\max}\lam^{k+1})) & \le & \Lk^d(U(\partial A_i, d_{\max}\lam^{k+1})) +
\Lk^d(U(\partial \Om_n,d_{\max}\lam^{k+1})) \\
& \le & C(\partial A_i,1) d_{\max}\lam^{k+1} \\
& + & 2(1+2b_1)^{d-1} d_{\max}\lam^{k+1} \Hk^{d-1}(\partial \Om_n).
\end{eqnarray*}
It is clear that $\Hk^{d-1}(\partial \Om_n)$ are uniformly bounded in $n$, since $A_i-y$ is polyhedral, and $\Om_n$ is its
approximation by a union of $2^{-n}r$-grid cubes. It follows that 
$$
\Lk^d(U(\partial V_n,d_{\max}\lam^{k+1}))\le b_2 \lam^k,\ \ \forall\,k\in \Z,\ k\le k_0,
$$ 
hence, by (\ref{eq-ult1}),
$$
\#\Rk^{(k)}(V_n) \le b_3 \lam^{-(d-1)k},\ \ \forall\,k\in \Z,\ k\le k_0.
$$
Finally, by (\ref{eq-ult2}) and (\ref{eq-ult3}),
\begin{eqnarray*}
|\Phi^+_{2,\Tk}(V_n)| & \le & b_4 \sum_{k=-\infty}^{k_0} \lam^{-(d-1)k}\theta_2^k \\
                      & \le & b_5 \left(\frac{\theta_2}{\lam^{d-1}}\right)^{k_0} 
                       \le  b_5 (2^{-n}r\sqrt{d})^{\alpha-(d-1)}.
\end{eqnarray*}
Since the latter tends to zero as $n\to \infty$ we obtain that 
$$
\Phi^+_{2,\Tk}(A_i-g)=\Phi^+_{2,\Tk}(V_n)+\Phi^+_{2,\Tk}(\Om_n) = \Phi^+_{2,\Tk}(V_n)=0, 
$$which is a contradiction. The theorem is proved completely.
\qed

\medskip

\noindent {\bf Acknowledgement.}
This project was started when both of us were visiting
MPIM Bonn, for whose warm hospitality we are deeply grateful.

\end{document}